\documentclass[11pt]{article}
\usepackage[colorlinks=true,
urlcolor=blue,linkcolor=blue,citecolor=blue]{hyperref}
\usepackage{amsmath,amsfonts,amssymb,amsthm}
\usepackage{mathrsfs}
\usepackage{color}
\usepackage{amscd}

\newtheorem{theorem}{Theorem}[section]
\newtheorem{lemma}[theorem]{Lemma}

\newtheorem{proposition}[theorem]{Proposition}

\theoremstyle{remark}

\numberwithin{equation}{section}

\def\eps{\varepsilon}

\def\bxi{ {\boldsymbol\xi} }
\def\bzeta{ {\boldsymbol\zeta} }

\title{Homogenization of the Schr{\"o}dinger equation with large, random potential}
\author{Ningyao Zhang \thanks{Department of Applied Physics and
        Applied Mathematics, Columbia University,
        New York NY, 10027; nz2164@columbia.edu} \and Guillaume Bal\thanks{Department of Applied Physics and
        Applied Mathematics, Columbia University,
        New York NY, 10027; gb2030@columbia.edu}}
\begin{document}
\maketitle

\begin{abstract}
\noindent We study the behavior of solutions to a Schr{\"o}dinger equation with large, rapidly oscillating, mean zero, random potential with Gaussian distribution. We show that in high dimension $d>\mathfrak{m}$, where $\mathfrak{m}$ is the
order of the spatial pseudo-differential operator in the Schr{\"o}dinger equation (with $\mathfrak{m}=2$ for the standard Laplace operator), the solution converges in the $L^2$ sense uniformly in time over finite intervals
to the solution of a deterministic Schr{\"o}dinger equation as the correlation length $\varepsilon$ tends to $0$. This generalizes to long times the convergence results obtained for short times and for the heat equation in \cite{B-MMS-10}. The result is based on the decomposition of multiple scattering contributions introduced in \cite{Erdos-Yau2}. In dimension $d<\mathfrak{m}$, the random solution converges to the solution of a stochastic partial differential equation; see \cite{B-CMP-09,ZB-ST-12}.
\end{abstract}

\section{Introduction}

There is a long list of derivations of macroscopic models for solutions to equations involving small scale heterogeneities $\eps\ll1$. One very successful framework is that of homogenization theory. In the limit  $\eps\to0$, it is shown that the heterogeneous solution converges to the deterministic solution of an effective medium equation \cite{BLP-78,JKO-SV-94}. There are cases, however, where the solutions in the limit $\eps\to0$ remain stochastic, typically in low spatial dimensions; see e.g. \cite{B-CMP-09,FGPS-07,H-CPAM-11,PP-GAK-06}.

The results of \cite{B-CMP-09,PP-GAK-06} apply to parabolic equations of the form of a heat equation with large, mean-zero, highly oscillatory, random potentials (zeroth-order terms). In \cite{B-MMS-10}, it was shown that for large dimensions, the random solution converged to a deterministic solution, which is consistent with the homogenization framework. However, such results could only be obtained for short times, and it is unclear whether they hold for larger times.

In this paper, we revisit the homogenization limit for the (time-dependent) Schr\"odinger equation. Because the solution operator is unitary in this case, we expect to be able to control the long-time asymptotic behavior of the solution. The method of proof, as in \cite{B-CMP-09,B-MMS-10}, is based on a Duhamel expansion of the random solution in terms involving increasing numbers of scattering events. As the number of terms grows exponentially with the number of scattering events, we need to assume that the potential is Gaussian in order to control such a growth. The summation technique used in \cite{B-MMS-10} cannot extend to long time controls. A similar difficulty occurs in the derivation of radiative transfer equations for the energy density of high frequency waves propagating in highly oscillatory media \cite{BKR-KRM-10,Erdos-Yau2, Spohn}. A precise summation of the scattering terms was introduced in \cite{Erdos-Yau2} to allow for long-time expansions. We adapt this technique to the asymptotic analysis of Schr\"odinger equations with large potentials.

\medskip

We now present in more detail the model considered in this paper and the main convergence result.
Let $\mathfrak{m}\geq 2$. We consider the following
Schr{\"o}dinger equation in dimension $d>\mathfrak{m}$:
\begin{equation}
\label{eq1}
\left\{
\begin{aligned}
\Big(i\frac{\partial}{\partial t}+\big(P(D)-\frac{1}{\varepsilon^{\mathfrak{m}/2}}q(\frac{x}{\varepsilon})\big)\Big)u_\varepsilon(t,x)&=0, \qquad t>0,\, x\in \mathbb{R}^d\\
u_\varepsilon(0,x)&=u_0(x), \qquad x\in \mathbb{R}^d.
\end{aligned}
\right.
\end{equation}
Here, $P(D)$ is the pseudo-differential operator with symbol $\hat{p}(\xi)=|\xi|^{\mathfrak{m}}$. We assume that $q(x)$ is a real valued mean zero stationary Gaussian process
defined on a probability space $(\Omega,\mathcal{F},\mathbb{P})$,
with correlation function $R(x)=\mathbb{E}\{q(y)q(x+y)\}$, and
the non-negative power spectrum $\hat{R}(\xi)$ is radially symmetric, smooth, and decays fast. For simplicity, we assume $\hat{R}\in\mathcal{S}(\mathbb{R}^d)$. In fact,  $\|R\|_{12d,12d}<+\infty$, where
\begin{equation}
\begin{aligned}
\|f\|_{d_1,d_2}&:=\|\langle x \rangle^{d_1} \langle \nabla_x \rangle^{d_2} f(x)\|_2,\qquad
\langle x \rangle:=(1+x^2)^{1/2},
\end{aligned}
\end{equation}
is enough. We choose the initial condition $u_0(x)$ to be smooth such that
$\hat{u}_0(\xi)\langle\xi\rangle^{6d}\in L^2(\mathbb{R}^d)$.\\
For any finite time $T>0$, the existence of a weak solution
$u_{\varepsilon}(t,x)\in L^2(\Omega\times\mathbb{R}^d)$ uniformly in
time $t\in(0,T)$ and $0<\varepsilon<\varepsilon_0$ can be proved by
using a method based on Duhamel expansion.\\
As $\varepsilon\rightarrow 0$, we show that the solution
$u_{\varepsilon}(t)$ to (1.1) converges strongly in
$L^2(\Omega\times\mathbb{R}^d)$ uniformly in $t\in(0,T)$ to its
limit $u(t)$ solution of the following homogenized equation
\begin{equation}\label{uhom}
\left\{
\begin{aligned}
(i\frac{\partial}{\partial t}+P(D)-\rho)u(t,x)&=0,\qquad t>0,\, x\in \mathbb{R}^d\\
u(0,x)&=u_0(x),\qquad x\in \mathbb{R}^d,
\end{aligned}
\right.
\end{equation}
where the potential is given by
\begin{equation}
\rho=\int_{\mathbb{R}^d}\frac{\hat{R}(\xi)}{|\xi|^{\mathfrak{m}}}d\xi.
\end{equation}
The main result of this paper is the following convergence result:
\begin{theorem}\label{thm}
There exists a solution to (1.1) $u_{\varepsilon}(t)$ uniformly in
$0<\varepsilon<\varepsilon_0$ for $t>0$. Moreover, we have the
convergence results for all $t\in (0,T)$
\begin{equation}
\label{thm1.1} \displaystyle\lim_{\varepsilon\rightarrow
0}\mathbb{E}\|(u_{\varepsilon}-u)(t)\|_2^2=0.
\end{equation}
\end{theorem}

The rest of the paper is organized as follows. Section 2 recasts
the solution to \eqref{eq1} as a Duhamel series expansion in the frequency
domain. We estimate the $L^2$ norm of the first $n_0$ terms by calculating the
contributions of graphs in three categories similar to those
defined in \cite{B-MMS-10}. Section 3 estimates the $L^2$ norm of the error term
by first subdividing the time integration into time intervals of smaller sizes, and then using
Duhamel formula in each time interval. This method, introduced in \cite{Erdos-Yau2},
significantly improves the error estimates compared to the direct estimates of infinite Duhamel terms and enables the elimination of the restriction to short times. The estimates given in these sections are used in section 4 to characterize the limit of the solution $u_{\varepsilon}(t,x)$. Section 5 provides the proofs for the inequalities used for justifying the estimates in the previous sections.

The analysis of a parabolic equation of the type of the heat equation (with $i\partial_t$ replaced by $\partial_t$) is performed in \cite{B-MMS-10} for $d\geq \mathfrak{m}$. Up to a logarithmic correction, we expect the limit of the solution to \eqref{eq1} to be deterministic also for the critical dimension $\mathfrak{m}=d$. In \cite{B-MMS-10}, the random fluctuations about the deterministic limit are also analyzed for the heat equation. We expect a similar behavior to occur for the Schr\"odinger equation \eqref{eq1} for short times. We do not know the behavior of the random fluctuations for arbitrary times $t\in (0,T)$.

In lower spatial dimension $d<\mathfrak{m}$, the limit of the solutions to \eqref{eq1} as $\eps\to0$ remains stochastic. This behavior was analyzed for the heat equation in \cite{B-CMP-09,KN-PA-10}. The limit of $u_\eps$ is then shown to be the solution of a stochastic partial differential equation with multiplicative noise (written as a Stratonovich product). The analysis of \eqref{eq1} for $d<\mathfrak{m}$ is performed in \cite{ZB-ST-12}. Note that several results of convergence may be extended to the case of random potential with long range correlations or random potentials that display both temporal and spatial fluctuations \cite{B-AMRX-11}.

\section{Duhamel expansion}
We denote by $e^{itH}$ the propagator for the equation (1.1). The
Duhamel expansion then states that for any $n_0\geq 1$
\begin{equation}
u_{\varepsilon}(t)=e^{itH}u_0=\sum_{n=0}^{n_0-1}u_{n,\varepsilon}(t)+\Psi_{n_0,\varepsilon}(t),
\end{equation}
where for $H_0:=(-\Delta)^{\frac{\mathfrak{m}}{2}}$, we have defined
\begin{equation}
\label{Duhamel} u_{n,\varepsilon}(t):=(-i)^n
(\frac{1}{\varepsilon^{\mathfrak{m}/2}})^n
\int_0^t\cdots\int_0^t(\prod_{k=0}^n ds_k)
\delta(t-\sum_{k=0}^{n}s_k) e^{i s_0
H_0}q(\frac{x}{\varepsilon})\cdots q(\frac{x}{\varepsilon})e^{i s_n
H_0}u_0,
\end{equation}
and the error term is given by
\begin{equation}
\Psi_{n_0,\varepsilon}(t)=(-i)\frac{1}{\varepsilon^{{\mathfrak{m}}/{2}}}\int_0^t
ds e^{i(t-s)H}q(\frac{x}{\varepsilon})u_{n_0-1,\varepsilon}(s).
\end{equation}

We shall choose
\begin{equation}
n_0=n_0(\varepsilon):=\frac{\gamma|\log{\varepsilon}|}{\log{|\log{\varepsilon}|}},
\end{equation}
for some fixed $0<\gamma\ll\lambda$ sufficiently small, where $\lambda$ is defined as
\begin{equation}
\lambda=\left\{
\begin{array}{ll}
d-\mathfrak{m} & \mathfrak{m}<d\leq 2\mathfrak{m}\\
\mathfrak{m} & d>2\mathfrak{m}.
\end{array}
\right.
\end{equation}
For any subset $I\subset
\mathbb{N}$, we define the kernel for the evolution in the Fourier
space as
\begin{equation}
K(t;\boldsymbol\xi,I):=(-i)^{|I|-1}\int_0^t(\prod_{k\in
I}ds_k)\delta(t-\sum_{k\in I}s_k)\prod_{k\in I} e^{i s_k
\xi_k^{\mathfrak{m}}}.
\end{equation}
Hereafter, we use the notation
$\xi^{\mathfrak{m}}=|\xi|^{\mathfrak{m}}$. In the special case
where $I=\{0,\cdots,n\}$, we denote by $K(t;\boldsymbol\xi,n):=K(t;\boldsymbol\xi,\{0,\cdots,n\})$ and
$\boldsymbol\xi_n := \boldsymbol\xi_{I_n}$. Denote
$\boldsymbol\xi_{n,\hat{0}}=\boldsymbol
\xi_{I_n\backslash\{0\}}$.\\
Let us introduce
$\hat{q}_{\varepsilon}(\xi)=\varepsilon^{d-\frac{\mathfrak{m}}{2}}\hat{q}(\varepsilon\xi)$,
the Fourier transform of
$\varepsilon^{-\frac{\mathfrak{m}}{2}}q(\frac{x}{\varepsilon})$.
Denote the contribution from the potential term by
\begin{equation}
L(\boldsymbol\xi,n):=\prod_{k=1}^{n}\hat{q}_{\varepsilon}(\xi_k-\xi_{k-1}).
\end{equation}
We may rewrite the $n^{th}$ order wave function as
\begin{equation}
\hat{u}_{\varepsilon,n}(t,\xi_0)=\int
K(t;\boldsymbol\xi,n)L(\boldsymbol\xi,n)\hat{u}_0(\xi_n)
d\boldsymbol\xi_{n,\hat{0}}.
\end{equation}
We need to introduce the following moments
\begin{equation}
U_{\varepsilon}^n(t,\xi_0)=\mathbb{E}\{\hat{u}_{\varepsilon,n}\},
\end{equation}
which are given by
\begin{equation}
\label{Uepsilon}
U_{\varepsilon}^n(t,\xi_0)=(-i)^n\int
K(t,\boldsymbol\xi,n)\mathbb{E}\{L(\bxi,n)\}\hat{u}_0(\xi_n)d\bxi_{n,\hat{0}},
\end{equation}
and
\begin{equation}
U_{\varepsilon}^{n,m}(t,\xi_0,\zeta_0)=\mathbb{E}\{\hat{u}_{\varepsilon,n}(t,\xi)\overline{\hat{u}_{\varepsilon,m}}(t,\zeta)\},
\end{equation}
which are given by
\begin{equation}
\begin{aligned}
U_{\varepsilon}^{n,m}(t,\xi_0,\zeta_0)=&(-i)^{n+m} \int
K(t,\bxi,n)\overline{K(t, \boldsymbol\zeta,m)}
\mathbb{E}\{L(\bxi,n)\overline{L(\boldsymbol\zeta,m)}\}
\hat{u}_0(\xi_n)\overline{\hat{u}_0(\zeta_m)}d\bxi_{n,\hat{0}}d\boldsymbol\zeta_{m,\hat{0}}.
\end{aligned}
\end{equation}

We need to estimate moments of the Gaussian process
$\hat{q}_{\varepsilon}$. The expectation in $U_{\varepsilon}^{n,m}$
vanishes unless there is $\bar{n}\in\mathbb{N}$ such that
$n+m=2\bar{n}$ is even. The moments are thus given as a sum of
products of the expectation of pairs of terms
$\hat{q}_{\varepsilon}(\xi_k-\xi_{k+1})$, where the sum runs over
all possible pairings. We define the pair $(\xi_k,\xi_l)$, $1\leq
k<l$, as the contribution in the product given by
\begin{equation}
\begin{aligned}
\mathbb{E}\{\hat{q}_{\varepsilon}(\xi_{k-1}-\xi_k)\hat{q}_{\varepsilon}(\xi_{l-1}-\xi_l)\}&=\varepsilon^{d-\mathfrak{m}}\hat{R}(\varepsilon(\xi_k-\xi_{k-1}))
\delta(\xi_k-\xi_{k-1}+\xi_l-\xi_{l-1})\\
&=\varepsilon^{d-\mathfrak{m}}r(\varepsilon(\xi_k-\xi_{k-1}))
r(\varepsilon(\xi_l-\xi_{l-1}))\delta(\xi_k-\xi_{k-1}+\xi_l-\xi_{l-1})
\end{aligned}
\end{equation}
with $r(\xi):=\hat{R}(\xi)^{1/2}$.

Define
\begin{equation}
F(\bxi,n):=\prod_{k=1}^{n} r(\varepsilon(\xi_k-\xi_{k-1}))
\hat{u}_0(\xi_n).
\end{equation}
Denote by $\Delta_{\pi}$ the product of delta functions associated
with the pairing $\pi$. Our analysis is based on the estimate of
\begin{equation}
U_{\varepsilon}^n(t,\xi)=\sum_{\pi\in\Pi(n)}I_{\pi}
\end{equation}
with
\begin{equation}
\label{I} I_{\pi}:=(-i)^n\int
K(t,\bxi,n)\Delta_{\pi}(\bxi)F(\bxi,n)d\bxi,
\end{equation}
and
\begin{equation}
\int
U_{\varepsilon}^{n,m}(t,\xi,\xi)d\xi=\sum_{\pi\in\Pi(n,m)}C_{\pi}
\end{equation}
with
\begin{equation}
\label{C} C_{\pi}:=(-i)^{n+m}\int
K(t,\bxi,n)\overline{K(t,\bzeta,m)}\Delta_{\pi}(\bxi,\bzeta)\delta(\xi_0-\zeta_0)F(\bxi,n)\overline{F(\bzeta,m)}d\boldsymbol\xi_n
d\boldsymbol\zeta_n.
\end{equation}
By Lemma 5.1, $K(t;\boldsymbol\xi,I)$ can also be written as
\begin{equation}
K(t;\boldsymbol\xi,I)=ie^{t\eta}\int d\alpha e^{-i\alpha t}
\prod_{k\in I} \frac{1}{\alpha+\xi_k^{\mathfrak{m}}+i\eta}.
\end{equation}
We let $\eta=t^{-1}$ in this section. Therefore, $I_{\pi}$ in \eqref{I} and $C_{\pi}$ in \eqref{C} can be written explicitly as
\begin{equation}
I_{\pi}=
-i^{n+1}e^{t\eta}\int d\alpha e^{-i\alpha t}\displaystyle\prod_{k=1}^{n} \frac{r(\varepsilon(\xi_k-\xi_{k-1}))}{\alpha+\xi_k^{\mathfrak{m}}+i\eta}\Delta_{\pi}(\xi)\hat{u}_0(\xi_n)d{\bxi}_{n,\hat{0}},
\end{equation}
and
\begin{equation}
\label{cpi}
\begin{aligned}
C_{\pi}=
-i^{n+m}e^{2t\eta}\int &d\alpha d\beta e^{-i(\alpha-\beta)t} \prod_{k=0}^{n}\frac{1}{\alpha+\xi_k^{\mathfrak{m}}+i\eta}
\prod_{l=0}^{m}\frac{1}{\beta+\zeta_l^{\mathfrak{m}}+i\eta}\prod_{k=1}^{n}r(\varepsilon(\xi_k-\xi_{k-1}))\\
&\times\prod_{l=1}^{m}r(\varepsilon(\zeta_l-\zeta_{l-1})\Delta_{\pi}(\bxi,\bzeta)\hat{u}_0(\xi_n)\overline{\hat{u}(\zeta_m)}\delta(\xi_0-\zeta_0)d\bxi_n d\bzeta_m.
\end{aligned}
\end{equation}
In order to consider the two sets of momenta in a unified way, we introduce the notation
\begin{equation}
\alpha_k=\left\{
\begin{array}{ll}
\alpha & 0\leq k \leq n\\
\beta & n+1\leq k \leq n+m+1,
\end{array}
\right.
\end{equation}
and define $\xi_{n+k+1}=\zeta_{m-k}$ for $0\leq k \leq m$. \eqref{cpi} can then be rewritten as
\begin{equation}
\begin{aligned}
C_{\pi}=-i^{n+m}e^{2t\eta}\int &d\alpha d\beta e^{-i(\alpha-\beta)t}\prod_{k=0}^{n+m+1} \frac{1}{\alpha_k+\xi_k^{\mathfrak{m}}+i\eta}\prod_{k=1,k\neq n+1}^{n+m+1} r(\varepsilon(\xi_k-\xi_{k-1}))\\
&\hat{u}_0(\xi_n)\overline{\hat{u}_0(\xi_m)}\delta(\xi_0-\xi_{n+m+1})d\bxi_{n+m+1}.
\end{aligned}
\end{equation}
In each instance of the pairings, we have $\bar{n}$ terms $k$ and
$\bar{n}$ terms $l\equiv l(k)$. Note that $l(k)\geq k+1$. The
collection of pairs $(\xi_k,\xi_{l(k)})$ for $\bar{n}$ values of $k$
and $\bar{n}$ values of $l(k)$ constitutes a graph $\pi\in\Pi(n,m)$,
where $\Pi(n,m)$ denotes the set of all graphs $\pi$ with $n$ copies
of $\hat{q}$ and $m$ copies of $\bar{\hat{q}}$. The graphs are
defined similarly in the calculation of $U_{\varepsilon}^n(t,\xi_0)$
in \eqref{Uepsilon} for $n=2\bar{n}$, and we denote by $\Pi(n)$ the set
of graphs with $n$ copies of $\hat{q}$. We denote by
$A_0=A_0(\mathfrak{\pi})$ the collection of the $\bar{n}$ values of
$k$ and by $B_0=B_0(\mathfrak{\pi})$ the
collection of the $\bar{n}$ values of $l(k)$.

Now we introduce several classes of graphs for $C_{\pi}$. We say
that the graph has a crossing if there is a $k\leq n$ such that
$l(k)\geq n+2$. We denote by $\Pi_c(n,m)\subset\Pi(n,m)$ the set of
graphs with at least one crossing and by $\Pi_{nc}(n,m)=\Pi(n,m)
\backslash \Pi_c(n,m)$ the non-crossing graphs.
We denote by simple
pairs the pairs such that $l(k)=k+1$, which thus involve a delta
function of the form $\delta(\xi_{k+1}-\xi_{k-1})$.
The unique
graph with only simple pairs is called
the simple graph, which is denoted by $\Pi_s(n,m)$. $\Pi_{ncs}(n,m)=\Pi_{nc}(n,m)\backslash\Pi_{s}(n,m)$ denotes the set of non-crossing, non-simple graphs.

We also use the notation $\Pi_s(n)$ and $\Pi_{ncs}(n)$ defined for $I_{\pi}$, which denote the simple graph and the set of non-crossing, non-simple graphs, respectively.

We shall estimate $F(\bxi,n)$ before we proceed to analyze the
graphs. For the initial condition, we define
\begin{equation}
\Phi(\xi)=\langle \xi \rangle ^{6d}|\hat{u}_0(\xi)|.
\end{equation}
From our assumption on $u_0$ given in Section 1, we have $\Phi(\xi)\in L^2(\mathbb{R}^d)$.
Since $\|\hat{R}\|_{12d,12d}<+\infty$, we have
\begin{equation}
r(\xi)\leq \frac{C}{\langle \xi \rangle^{6d}}.
\end{equation}
Hence, we obtain the estimate for $F(\bxi,n)$
\begin{equation}
\label{F}
\begin{aligned}
|F(\bxi,n)|\leq &\prod_{k=1}^{n}
\frac{1}{\langle\varepsilon(\xi_k-\xi_{k-1})\rangle^{6d}}\frac{\Phi(\xi_n)}{\langle
\xi_n \rangle^{6d}}\\
\leq &\prod_{k=1}^n \frac{1}{\langle
\varepsilon(\xi_k-\xi_{k-1})\rangle^{6d}}
\frac{\Phi(\xi_n)}{\langle \xi_n \rangle^{2d}\langle
\varepsilon\xi_n
\rangle^{4d}}\\
\leq &C^n \Phi(\xi_n) \frac{1}{\langle
\xi_{n}\rangle^{2d}}\prod_{i=1}^{2} \frac{1}{\langle \varepsilon
\xi_{l_i}\rangle^{2d}} \prod_{k=1}^n
\frac{1}{\langle\varepsilon(\xi_k-\xi_{k-1})\rangle^{2d}}.
\end{aligned}
\end{equation}
We have the freedom to choose $l_1$ and $l_2$ between $0$ and
$n$.\\

\noindent {\bf Analysis of the crossing graphs}
\begin{lemma}
If $\pi\in\Pi_c(n,m)$, we have
\begin{equation}
|C_{\pi}|\leq (C\log\varepsilon)^{\bar{n}}
\varepsilon^{\lambda}\|\Phi(\xi)\|_2^2.
\end{equation}
\end{lemma}
\begin{proof}
Denote by $(\xi_{q_m},\xi_{l(q_m)})$, $1\leq m\leq M$, the crossing
pairs and define $Q=\max_m\{q_m\}$. Let us define
$A'=A_0\backslash\{Q\}$. From \eqref{F}, we have
\begin{equation}
|F(\bxi,n)\overline{F(\bzeta,m)}|\leq
|\Phi(\xi_n)||\Phi(\xi_{n+1})|\frac{1}{\langle \varepsilon\xi_0
\rangle^{2d}}\frac{1}{\langle \xi_n \rangle^{2d}}\prod_{k\in
A'}\frac{1}{\langle \varepsilon(\xi_k-\xi_{k-1}) \rangle^{2d}}.
\end{equation}
The terms $\frac{1}{|\alpha_k+\xi_k^{\mathfrak{m}}+it^{-1}|}$ for
$k\notin A'\cup\{0,n,n+1,n+m+1\}$ are bounded by $C$. This allows us
to obtain
\begin{equation}
\begin{aligned}
|C_{\pi}|\leq & C^{\bar{n}} \int d\alpha d\beta
\frac{1}{|\alpha+\xi_0^{\mathfrak{m}}+it^{-1}|}\frac{1}{|\beta+\xi_0^{\mathfrak{m}}+it^{-1}|}\frac{1}{\langle\varepsilon\xi_0\rangle^{2d}}\\
&\prod_{k\in A'}
\frac{1}{|\alpha_k+\xi_k^{\mathfrak{m}}+it^{-1}|}\frac{1}{\langle\varepsilon(\xi_k-\xi_{k-1})\rangle^{2d}}\delta(\xi_k-\xi_{k-1}+\xi_{l(k)}-\xi_{l(k)-1})\\
&\delta(\xi_Q-\xi_{Q-1}+\xi_{l(Q)}-\xi_{l(Q)-1})\\
&\frac{1}{|\alpha+\xi_n^{\mathfrak{m}}+it^{-1}|}\frac{1}{|\beta+\xi_n^{\mathfrak{m}}+it^{-1}|}
|\Phi(\xi_n)|^2 d\boldsymbol\xi_{n+m}.
\end{aligned}
\end{equation}
For each $k\in A'\cup \{0\}$, we perform the change of variables
$\xi_k\rightarrow\frac{\xi_k}{\varepsilon}$, and define
\begin{equation}
\xi_k^{\varepsilon}=\left\{
\begin{array}{ll}
\xi_k & k\ne A'\cup \{0\}\\
\frac{\xi_k}{\varepsilon} & k\in A'\cup \{0\}.
\end{array}
\right.
\end{equation}
We then find the estimate
\begin{equation}
\begin{aligned}
|C_{\pi}| \leq & C^{\bar{n}}\int d\alpha d\beta
\frac{1}{|\alpha+\xi_0^{\mathfrak{m}}/{\varepsilon^{\mathfrak{m}}}+it^{-1}|}\frac{1}{|\beta+\xi_0^{\mathfrak{m}}/{\varepsilon^{\mathfrak{m}}}+it^{-1}|}
\varepsilon^{-\mathfrak{m}}\frac{1}{\langle\xi_0\rangle^{2d}}\\
&\prod_{k\in
A'}\frac{1}{|\alpha_k+\xi_k^{\mathfrak{m}}/{\varepsilon^{\mathfrak{m}}}+it^{-1}|}\varepsilon^{-\mathfrak{m}}\frac{1}{\langle\xi_k-\varepsilon\xi_{k-1}^{\varepsilon}\rangle^{2d}}\\
&\delta(\frac{\xi_k}{\varepsilon}-\xi_{k-1}^{\varepsilon}+\xi_{l(k)}-\xi_{l(k)-1}^{\varepsilon})\delta(\xi_Q-\xi_{Q-1}+\xi_{l(Q)}-\xi_{l(Q)-1})\\
&\frac{1}{|\alpha+\xi_n^{\mathfrak{m}}+it^{-1}|}\frac{1}{|\beta+\xi_n^{\mathfrak{m}}+it^{-1}|}\frac{1}{\langle
\xi_{n}\rangle^{2d}}|\Phi(\xi_n)|^2 d\boldsymbol\xi_{n+m}.
\end{aligned}
\end{equation}
We now estimate the above product. Assume $Q<n$ and $n=l(k_0)$. When
$Q=n$ or $l(Q)=n+m+1$, the derivation of the same estimates is simpler and left to the reader.
Define $k_1$ such that $l(k_1)=n+m+1$. For each $k\in
A'(\mathfrak{\pi})\backslash(k_0\cup k_1)$, we use \eqref{5.2.4} below to
find the estimate
\begin{equation}
\varepsilon^{-\mathfrak{m}} \int
\frac{1}{\langle\xi_k-\varepsilon\xi_{k-1}^{\varepsilon}\rangle^{2d}}\frac{1}{|\alpha_k+\xi_k^{\mathfrak{m}}/{\varepsilon^{\mathfrak{m}}}+it^{-1}|}
\delta(\frac{\xi_k}{\varepsilon}-\varepsilon\xi_{k-1}^{\varepsilon}+\xi_{l(k)}-\xi_{l(k)-1}^{\varepsilon})d\xi_kd\xi_{l(k)}\leq
C|\log\varepsilon|.
\end{equation}
The integration in the $\xi_{l(Q)}$ variable is estimated by using
the above delta function. The delta function for $k=k_0\in
A'(\mathfrak{\pi})$ may be written in the form
$\delta(\xi_Q-\xi_{Q-1}+\xi_0-\xi_n+\sum_{m=1}^{M-1}\xi_{q_m}-\xi_{q_m-1})$,
and is thus used to integrate in the variable $\xi_Q$. The term
$\frac{1}{\langle\xi_{k_0}-\varepsilon\xi_{k_0-1}^{\varepsilon}\rangle^{2d}}$
is used to integrate in the variable $\xi_{k_0}$. The integral in
$\alpha$ and $\beta$ is estimated using \eqref{5.2.1} by
\begin{equation}
\int
\frac{1}{|\alpha+\xi_0^{\mathfrak{m}}/{\varepsilon^{\mathfrak{m}}}+it^{-1}|}\frac{1}{|\alpha+\xi_n^{\mathfrak{m}}+it^{-1}|}d\alpha
\leq
C\frac{1}{|\xi_0^{\mathfrak{m}}/{\varepsilon^{\mathfrak{m}}}-\xi_n^{\mathfrak{m}}+it^{-1}|}\left[1+\log_{+}\left|\frac{\xi_0^{\mathfrak{m}}/{\varepsilon^{\mathfrak{m}}}-\xi_n^{\mathfrak{m}}}{t^{-1}}\right|\right].
\end{equation}
The integral in $\xi_0$ is estimated using \eqref{5.1} by
\begin{equation}
\varepsilon^{-\mathfrak{m}}\int
\frac{1}{\langle\xi_0\rangle^{2d}}\frac{\left[1+\log_{+}\left|\frac{\xi_0^{\mathfrak{m}}/{\varepsilon^{\mathfrak{m}}}-\xi_n^{\mathfrak{m}}}{t^{-1}}\right|\right]^2}{|\xi_0^{\mathfrak{m}}/{\varepsilon^{\mathfrak{m}}}-\xi_n^{\mathfrak{m}}+it^{-1}|^2}d\xi_0
\leq C\varepsilon^{\lambda}(\xi_n^{\lambda}\vee 1).
\end{equation}
Following the usual convention, we use $a\vee b :=\max\{a,b\}$ and $a \wedge b:= \min\{a,b\}$.
The delta function
$\delta(\xi_{k_1}-\xi_{k_1-1}-\xi_{n+m+1}-\xi_{n+m})$ is seen to be
equivalent to $\delta(\xi_n-\xi_{n+1})$, which handles the
integration in the variable $\xi_{n+1}$. Finally we integrate in
$\xi_n$
\begin{equation}
\int \frac{1}{\langle\xi_n\rangle^{2d}}|(\xi_n^{\lambda}\vee
1)\Phi(\xi_n)|^2d\xi_n\leq \|\Phi(\xi)\|_2^2,
\end{equation}
and obtain
\begin{equation}
|C_{\pi}| \leq
(C|\log{\varepsilon}|)^{\bar{n}}\varepsilon^{\lambda}\|\Phi(\xi)\|_2^2.
\end{equation}
\end{proof}
\noindent Using Stirling's formula, we find that $|\Pi_c(n,m)|\leq
\frac{(2\bar{n}-1)!}{2^{\bar{n}-1}(\bar{n}-1)!}$ is bounded by
$\left(\frac{2\bar{n}}{e}\right)^{\bar{n}}$. After summation in
$n,m\leq n_0$, we obtain
\begin{equation}
\label{est2.1}
\begin{aligned}
\mathbb{E}\|\sum_{n=0}^{n_0}(\hat{u}_{n,\varepsilon}-U_{\varepsilon}^n)(t)\|_2^2
\leq
&\sum_{n=0}^{n_0-1}\sum_{m=0}^{n_0-1}(\frac{2\bar{n}}{e})^{\bar{n}}(C|\log\varepsilon|)^{\bar{n}}
\varepsilon^{\lambda}\|\Phi(\xi)\|_2^2\\
\leq &n_0^{n_0} (C|\log\varepsilon|)^{n_0}
\varepsilon^{\lambda}\|\Phi(\xi)\|_2^2 \lesssim
\varepsilon^{\lambda-3\gamma},
\end{aligned}
\end{equation}
where $a\lesssim b$ means $a\leq Cb$ for some $C>0$.

\noindent {\bf Analysis of the non-crossing graphs}
\begin{lemma}
If $\pi\in \Pi_{ncs}(n)$, we have
\begin{equation}
\label{2.2.1} |I_{\pi}|\leq
(C\log\varepsilon)^{\frac{n}{2}}\varepsilon^{\lambda(1-\delta)}|\Phi(\xi)|,
\end{equation}
where $0<\delta\ll 1$.
Moreover, if $\pi\in \Pi_{ncs}(n,m)$,
\begin{equation}
\label{2.2.2} |C_{\pi}|\leq
(C\log\varepsilon)^{\bar{n}}\varepsilon^{2\lambda(1-\delta)}\|\Phi(\xi)\|_2^2.
\end{equation}
\end{lemma}
\begin{proof}
 In a graph $\pi\in \Pi_{ncs}(n)$, the delta function
\begin{equation}
\label{ncsdelta}
\delta(\xi_0-\xi_n)
\end{equation}
 is obtained by adding up all the delta functions in $\Delta_{\pi}$. We perform the change of variables for all $k\in A_0$, $\xi_k\rightarrow
\frac{\xi_k}{\varepsilon}$ and define as before
\begin{equation}
\xi_k^{\varepsilon}=\left\{
\begin{array}{ll}
\xi_k & k\notin A_0\\
\frac{\xi_k}{\varepsilon} & k\in A_0.
\end{array}
\right.
\end{equation}
We shall solve the following two cases in different ways.

(i) If there exists $k_2\in A_0$, such that for all $k$ satisfying
$k_2+1\leq k\leq l(k_2)-2, k\in A_0$, $(\xi_k,\xi_{l(k)})$ are
simple pairs, then the delta function
$\delta(\frac{\xi_{k_2}}{\varepsilon}-\xi_{l(k_2)-1})$ is present.
From \eqref{F}, we have
\begin{equation}
\label{ncsF} F(\boldsymbol\xi,n)\leq
C^n|\Phi(\xi_n)|\frac{1}{\langle
\xi_{k_2}\rangle^{2d}}\frac{1}{\langle \xi_n \rangle^{2d}}\prod_{k\in
A_0}\frac{1}{\langle\xi_k-\varepsilon\xi_{k-1}^{\varepsilon}
\rangle^{2d}}.
\end{equation}
The estimate of integration in
$\xi_{k_2}$ is then obtained by using \eqref{5.1} below:
\begin{equation}
\begin{aligned}
\int\varepsilon^{-\mathfrak{m}}\frac{1}{\langle\xi_{k_2}\rangle^{2d}}\frac{1}{|\alpha+\xi_{k_2}^{\mathfrak{m}}/\varepsilon^{\mathfrak{m}}+it^{-1}|^2}d\xi_{k_2}
\leq C\varepsilon^{\lambda(1-\delta)}(\alpha^{\frac{\lambda(1-\delta)}{\mathfrak{m}}}\vee
1).
\end{aligned}
\end{equation}
The delta function in which $\xi_n$ is involved is equivalent to
$\delta(\xi_n-\xi_0)$, which we use to integrate in $\xi_n$:
\begin{equation}
\label{deltan} \int
\frac{1}{\langle\xi_n\rangle^{2d}}\frac{1}{|\alpha+\xi_n^{\mathfrak{m}}+it^{-1}|}\delta(\xi_n-\xi_0)d\xi_n\leq
\frac{1}{\langle\xi_0\rangle^{2d}}\frac{1}{|\alpha+\xi_0^{\mathfrak{m}}+it^{-1}|}.
\end{equation}
For $k\in A_0, k\ne k_2$, we have the estimate
\begin{equation}
\label{simple} \int
\varepsilon^{-\mathfrak{m}}\frac{1}{\langle\xi_k-\varepsilon\xi_{k-1}^\varepsilon\rangle^{2d}}\frac{1}{|\alpha+\xi_k^{\mathfrak{m}}/{\varepsilon^{\mathfrak{m}}}+it^{-1}|}\delta(\frac{\xi_k}{\varepsilon}-\xi_{k-1}^{\varepsilon}+\xi_{l(k)}-\xi_{l(k)-1}^{\varepsilon})d\xi_kd\xi_{l(k)}\leq
C|\log\varepsilon|.
\end{equation}
The estimate of integration in the variable $\alpha$ is then given by
\begin{equation}
\int\frac{\alpha^{\frac{\lambda(1-\delta)}{\mathfrak{m}}}\vee
1}{|\alpha+\xi_0^{\mathfrak{m}}+it^{-1}|^2}\leq C
\xi_0^{\mathfrak{\lambda(1-\delta)}}.
\end{equation}
The extra term $\xi_0^{\lambda(1-\delta)}$ that arises in the last estimate can be
canceled by the term ${1}/{\langle\xi_0\rangle^{2d}}$ in
\eqref{deltan},
which concludes \eqref{2.2.1}.

(ii) If there exists no such $k_2\in A_0$ satisfying the condition
in case (i), then we first delete all simple pairs that exist in the
graph $\pi$.
In fact, the simple pairs can be handled first by using the bound as
in \eqref{simple}. Therefore without loss of generality, we need
only to consider a graph $\pi$ with no simple pair. Let us define
$k_4=\displaystyle\min\{k|k\in A_0,l(k)-1\in A_0\}$, and
$k_5=l(k_4)-1$. Note that $k_5\geq k_4+1$.

We have from \eqref{F} that
\begin{equation}
\label{ncsF2} F(\boldsymbol\xi,n)\leq
C^n|\Phi(\xi_n)|\frac{1}{\langle
\xi_{l(k_4)}\rangle^{2d}}\frac{1}{\langle \xi_{k_5}
\rangle^{2d}}\frac{1}{\langle \xi_n \rangle^{2d}}\prod_{k\in
A_0}\frac{1}{\langle\xi_k-\varepsilon\xi_{k-1}^{\varepsilon}
\rangle^{2d}}.
\end{equation}
The integration in $\xi_{l(k_4)}$ provides the terms which we will
need for integration in $\xi_{k_5}$
\begin{equation}
\begin{aligned}
&\int\frac{1}{\langle\varepsilon\xi_{l(k_4)}\rangle^{2d}}\frac{1}{|\alpha+\xi_{l(k_4)}^{\mathfrak{m}}+it^{-1}|}\delta(\frac{\xi_{k_4}}{\varepsilon}-\xi_{k_4-1}^{\varepsilon}+\xi_{l(k_4)}-\frac{\xi_{k_5}}{\varepsilon})d\xi_{l(k_4)}\\
=&\frac{1}{\langle\xi_{k_5}-\xi_{k_4}+\varepsilon\xi_{k_4-1}^{\varepsilon}\rangle^{2d}}\frac{1}{|\alpha+|\xi_{k_5}-\xi_{k_4}+\varepsilon\xi_{k_4-1}^{\varepsilon}|^{\mathfrak{m}}/{\varepsilon^{\mathfrak{m}}}+it^{-1}|}.
\end{aligned}
\end{equation}
We can estimate the integration in $\xi_{k_5}$ using \eqref{5.1}
\begin{equation}
\begin{aligned}
\int
&\varepsilon^{-\mathfrak{m}}\frac{1}{\langle\xi_{k_5}\rangle^{2d}}\frac{1}{\langle\xi_{k_5}-\xi_{k_4}+\varepsilon\xi_{k_4-1}^{\varepsilon}\rangle^{2d}}
\frac{1}{|\alpha+\xi_{k_5}^{\mathfrak{m}}/{\varepsilon^{\mathfrak{m}}}+it^{-1}|}\\
&\frac{1}{|\alpha+|\xi_{k_5}-\xi_{k_4}+\varepsilon\xi_{k_4-1}^{\varepsilon}|^{\mathfrak{m}}/{\varepsilon^{\mathfrak{m}}}+it^{-1}|}d\xi_{k_5}\leq
C\varepsilon^{\lambda(1-\delta)}(\alpha^{\frac{\lambda(1-\delta)}{\mathfrak{m}}}\vee 1).
\end{aligned}
\end{equation}
All the other integrations are handled the same way as in case (i).
In order to make sure that no
integration above is affected by other integrands we plan to use for integrating in other variables, we just need to first integrate in $\xi_{l(k)}$ with index
in decreasing order and then integrate in $\xi_k$ with index in decreasing order. This gives \eqref{2.2.1}.\\
If $\pi\in\Pi_{ncs}(n,m)$, we may denote the pairings for $k\leq n$
and for $n+1\leq k \leq n+m+1$ by $\pi_1$ and $\pi_2$ and then
$\pi=\pi_1\cup\pi_2$, since there is no crossing in $\pi$. Hence it
follows that
\begin{equation}
|C_{\pi}|\leq \int |I_{\pi_1}(\xi) I_{\pi_2}(\xi)| d\xi \leq
(C\log\varepsilon)^{\bar{n}}\varepsilon^{2\lambda(1-\delta)}\|\Phi(\xi)\|_2^2.
\end{equation}
\end{proof}
Similarly to \eqref{est2.1}, we obtain
\begin{equation}
\label{est2.2}
\begin{aligned}
\|\sum_{n=0}^{n_0-1}(U_{\varepsilon}^n-U_{\varepsilon,s}^n)(t)\|_2^2\leq
&\sum_{n=0}^{n_0-1}\sum_{m=0}^{n_0-1}(\frac{2\bar{n}}{e})^{\bar{n}}(C|\log\varepsilon|)^{\bar{n}}
\varepsilon^{2\lambda(1-\delta)}\|\Phi(\xi)\|_2^2\\
\leq &n_0^{n_0} (C|\log\varepsilon|)^{n_0}
\varepsilon^{2\lambda(1-\delta)}\|\Phi(\xi)\|_2^2 \lesssim
\varepsilon^{2\lambda(1-\delta)-3\gamma},
\end{aligned}
\end{equation}
where
\begin{equation}
\label{simpledef}
U_{\varepsilon,s}^{n}:=I_{\pi}
\end{equation}
for $\pi\in\Pi_s(n)$.\\
Collecting the results obtained in \eqref{est2.1} and
\eqref{est2.2}, we have shown that
\begin{equation}
\label{est2.2.2}
\mathbb{E}\|\sum_{n=0}^{n_0-1}(\hat{u}_{\varepsilon,n}-U_{\varepsilon,s}^n)(t)\|_2^2\lesssim
\varepsilon^{{2\lambda(1-\delta)-3\gamma}}.
\end{equation}

\noindent {\bf Analysis of the simple Graphs}
\begin{lemma}
If $\pi\in \Pi_s(n)$, we have
\begin{equation}
\label{2.3.1} |I_{\pi}|\leq
(\frac{C^{n}}{(n/2)!}+O(C^{n}\varepsilon^{\mathfrak{m}/2}))|\hat{u}_0(\xi)|.
\end{equation}
Moreover, if $\pi\in \Pi_s(n,m)$, we have
\begin{equation}
\label{2.3.2} |C_{\pi}|\leq
(\frac{C^{n+m}}{(n/2)!(m/2)!}+O(C^{n+m}\varepsilon^{\mathfrak{m}/2}))\|\hat{u}_0(\xi)\|_2^2.
\end{equation}
\end{lemma}
\begin{proof}
In the case of a simple graph $\pi$, we can explicitly write out the
product of delta functions
\begin{equation}
\Delta_{\pi}=\prod_{k\in A_0}\delta(\xi_{k-1}-\xi_{k+1}),
\end{equation}
which is independent of $\xi_k$ for all $k\in A_0$, and forces
$\xi_k=\xi_0$ for all $k\notin A_0$.\\
Integrating in $\xi_k$ for all $k\in B_0$ using delta functions, we
obtain
\begin{equation}
I_{\pi}=\int
K(t,\xi_0,\cdots,\xi_0,\xi_{k_{a_1}},\cdots,\xi_{k_{a_{n/2}}})\prod_{k\in
A_0}
\varepsilon^{d-\mathfrak{m}}\hat{R}(\varepsilon(\xi_k-\xi_{0}))\hat{u}_0(\xi_0)
d\boldsymbol\xi_{A_0},
\end{equation}
where $\{k_{a_1},\cdots,k_{a_{n/2}}\}=A_0$.\\
This implies
\begin{equation}
\label{simple} |I_{\pi}|\leq \int d\alpha
\frac{1}{|\alpha+\xi_0^{\mathfrak{m}}+i\eta|^{n/2+1}}\prod_{k\in
A_0}\left|\int\frac{\varepsilon^{d-\mathfrak{m}}\hat{R}(\varepsilon(\xi_k-\xi_{0}))}{\alpha+\xi_k^{\mathfrak{m}}+i\eta}d\xi_k\right||\hat{u}_0(\xi_0)|.
\end{equation}
Define
\begin{equation}
\Theta_{\alpha,\eta}(\xi_0)=\int
\frac{\varepsilon^{d-\mathfrak{m}}\hat{R}(\varepsilon(\xi_k-\xi_{0}))}{\alpha+\xi_k^{\mathfrak{m}}+i\eta}d\xi_k.
\end{equation}
Perform the change of variable $\xi_k\rightarrow
\frac{\xi_k}{\varepsilon}$. This gives
\begin{equation}
\Theta_{\alpha,\eta}(\xi_0)=\int
\frac{\hat{R}(\xi_k-\varepsilon\xi_{0})}{\varepsilon^{\mathfrak{m}}\alpha+\xi_k^{\mathfrak{m}}+i\varepsilon^{\mathfrak{m}}\eta}d\xi_k.
\end{equation}
It is clear from Lemma 5.8 that
\begin{equation}
\label{Theta} |\Theta_{\alpha,\eta}(\xi_{0})|\leq C.
\end{equation}
Thus \eqref{simple} already implies
\begin{equation}
|I_{\pi}|\leq C^{n}|\hat{u}_0(\xi)|.
\end{equation}
However, this estimate is not sufficient and there is in fact a term
$1/{(n/2)!}$ missing. We now recover this factor.

Introduce the notation
\begin{equation}
\Theta(\xi)=\lim_{\eta\rightarrow 0}
\Theta_{\xi^{\mathfrak{m}},\eta}(\xi).
\end{equation}
From the estimate in \eqref{5.4.2}, we have
\begin{equation}
\label{ThetaErr} |\Theta_{\alpha,\eta}(\xi_0)-\Theta(\xi_k)|\leq
\varepsilon^{\mathfrak{m}/2}(|\alpha+\xi_{0}^{\mathfrak{m}}||\eta|^{-1/2}+|\eta|^{1/2}).
\end{equation}
We shall show that the leading term of $I_{\pi}$ is
\begin{equation}
\label{SimpleLead}
K(t,\xi_0,\cdots,\xi_0)\Theta(\xi_0)^{n/2}\hat{u}_0(\xi_0).
\end{equation}
In fact, the error term is bounded by
\begin{equation}
\label{SimpleErr} \int d\alpha
\frac{1}{|\alpha+\xi_0^{\mathfrak{m}}+i\eta|^{n/2+1}}|\Theta_{\alpha,\eta}^{n/2}(\xi_0)-\Theta^{n/2}(\xi_0)||\hat{u}_0(\xi_0)|.
\end{equation}
Using the uniform bound on $\Theta$ in \eqref{Theta} and the
estimate in \eqref{ThetaErr}, we can bound \eqref{SimpleErr}
by $O(\varepsilon^{\mathfrak{m}/2}C^n|\hat{u}_0(\xi_0)|)$.

We can now use Lemma 5.1 to bound \eqref{SimpleLead} by
\begin{equation}
\frac{C^{n}}{(n/2)!}|\hat{u}_0(\xi_0)|.
\end{equation}
Finally, we obtain \eqref{2.3.2} as an immediate consequence of \eqref{2.3.1}. This concludes
Lemma 2.3.
\end{proof}

\section{Partial Time Integration}
In this section, we estimate the $L^2$ norm of $\Psi_{n_0,\varepsilon}$. The central idea is to subdivide the time integration into
smaller time intervals of size $t/\kappa(\varepsilon)$ with
\begin{equation}
\kappa(\varepsilon):=|\log\varepsilon|^{1/\gamma^2}.
\end{equation}
We then use the Duhamel formula to estimate the evolution in each
time interval.\\
Recall the error term
$\Psi_{n_0,\varepsilon}=\sum_{n=n_0}^{+\infty}
u_{n,\varepsilon}$. The Duhamel formula states that
\begin{equation}
\Psi_{n_0,\varepsilon}(t)=(-i)\frac{1}{\varepsilon^{\mathfrak{m}/2}}\int_0^t
ds e^{i(t-s)H}q(\frac{x}{\varepsilon})u_{n_0-1,\varepsilon}(s),
\end{equation}
where $e^{itH}$ denotes the propagator of equation (1.1).\\
Let $\theta_j=jt/{\kappa}$ for $j=0,1,\cdots,\kappa$. Rewrite
\begin{equation}
\label{partial}
\Psi_{n_0,\varepsilon}(t)=(-i)\frac{1}{\varepsilon^{\mathfrak{m}/2}}\displaystyle\sum_{j=0}^{\kappa-1}e^{i(t-\theta_{j+1})H}\int_{\theta_{j}}^{\theta_{j+1}}e^{i(\theta_{j+1}-s)H}q(\frac{x}{\varepsilon})u_{n_0-1,\varepsilon}(s)ds.
\end{equation}
Define the n-th term of the Duhamel expansion for the operator
$e^{i(\theta_{j+1}-s)H}$ in \eqref{partial} as
\begin{equation}
\begin{aligned}
u_{n,n_0,\theta_j}=&(-i)^{n-n_0+1}(\frac{1}{\varepsilon^{\mathfrak{m}/2}})^{n-n_0+1}\int_{\theta_j}^{\theta_{j+1}}\int_0^t\cdots\int_0^t
(\prod_{k=0}^{n-n_0}ds_k)\delta(\theta_{j+1}-s-\sum_{k=0}^{n-n_0}s_k)\\
&e^{is_0H_0}q(\frac{x}{\varepsilon})\cdots
q(\frac{x}{\varepsilon})e^{is_{n-n_0+1}H_0}u_{n_0-1,\varepsilon}(s)ds.
\end{aligned}
\end{equation}
We may further obtain the form of $u_{n,n_0,\theta_j}$ in terms of
$u_0$ by writing $u_{n_0-1,\varepsilon}(s)$ out explicitly using
\eqref{Duhamel}
\begin{equation}
\begin{aligned}
u_{n,n_0,\theta_j}=&(-i)^{n}(\frac{1}{\varepsilon^{\mathfrak{m}/2}})^{n}\int_{\theta_j}^{\theta_{j+1}}\int_0^t\cdots\int_0^t(\prod_{k=0}^{n}ds_k)\\
&\delta(\theta_{j+1}-s-\sum_{k=0}^{n-n_0}s_k)\delta(s-\sum_{k=n-n_0+1}^{n}s_k)e^{is_0H_0}q(\frac{x}{\varepsilon})\cdots
q(\frac{x}{\varepsilon})e^{is_nH_0}u_0ds.
\end{aligned}
\end{equation}
The amputated versions of these functions are defined as
\begin{equation}
\tilde{u}_{4n_0,n_0,\theta_j}(x)=
\frac{1}{\varepsilon^{\mathfrak{m}/2}}q(\frac{x}{\varepsilon})u_{4n_0-1,n_0,\theta_j}(x).
\end{equation}
Then Duhamel formula then gives
\begin{equation}
\Psi_{n_0,\varepsilon}=U_1+U_2,
\end{equation}
where
\begin{equation}
\begin{aligned}
&U_1(t)=\sum_{n_0\leq n <4n_0}\sum_{j=0}^{\kappa-1} e^{i(t-\theta_{j+1})H}u_{n,n_0,\theta_j}(\theta_{j+1}),\\
&U_2(t)=(-i)\sum_{j=0}^{\kappa-1}e^{i(t-\theta_{j+1})H}\int_{\theta_j}^{\theta_{j+1}}\tilde{u}_{4n_0,n_0,\theta_j}(s)ds.
\end{aligned}
\end{equation}
From the unitarity of $e^{i(t-\theta_{j+1})H}$ and the triangle
inequality, we can bound $U_1$ by
\begin{equation}
\label{u1} \|U_1\|_2^2\leq C n_0 \kappa \displaystyle\sum_{n_0 \leq
n <
4n_0}\sum_{j=0}^{\kappa-1}\|u_{n,n_0,\theta_j}(\theta_{j+1})\|_2^2
\leq C n_0^2\kappa^2\sup_{n_0\leq n<4n_0}
\|u_{n,n_0,\theta_j}\|_2^2.
\end{equation}
Applying the Cauchy-Schwarz inequality, we can bound $U_2$ by
\begin{equation}
\label{u2} \|U_2\|_2^2 \leq t \sum_{j=0}^{\kappa-1}
\int_{\theta_j}^{\theta_{j+1}}\|\tilde{u}_{4n_0,n_0,\theta_j}(s)\|_2^2
ds \leq t^2\sup_j \sup_{\theta_j\leq s\leq \theta{j+1}}
\|\tilde{u}_{4n_0,n_0,\theta_j}(s)\|_2^2.
\end{equation}
Denote
\begin{equation}
I_{n-n_0+1,n}:=\{n-n_0+1,\cdots,n\}.
\end{equation}
Define the free evolution operator with constraint given by the
parameters $n_0$ and $\theta$ as
\begin{equation}
K^{\#}(\theta_{j+1},\theta_j;\bxi,n,n_0):=\int_{\theta_j}^{\theta_{j+1}}
K(\theta_{j+1}-s;\bxi,n-n_0)K(s,\bxi,I_{n-n_0+1,n})ds.
\end{equation}
We can write the wave function in Fourier space
$\hat{u}_{n,n_0,\theta_j}$ as
\begin{equation}
\hat{u}_{n,n_0,\theta_j}(\theta_{j+1},\xi_0)=\int
K^{\#}(\theta_{j+1},\theta_j;\bxi,n,n_0) L(\bxi,n)
\hat{u}_0(\xi_n)d\boldsymbol\xi_{n,\hat{0}}.
\end{equation}
We then write
\begin{equation}
\mathbb{E}\|u_{n,n_0,\theta_j}(\theta_{j+1})\|^2 = \sum_{\pi\in \Pi(n,n)}
C_{\pi}^{\#},
\end{equation}
where
\begin{equation}
C_{\pi}^{\#}:=\int d\boldsymbol\xi d\boldsymbol\zeta
K^{\#}(\theta_{j+1},\theta_j;\xi,n,n_0)
\overline{K^{\#}(\theta_{j+1},\theta_j;\zeta,n,n_0)} \Delta_{\pi}
F(\xi,n) \overline{F(\zeta,n)}.
\end{equation}
For the amputated function, we have
\begin{equation}
\mathbb{E}\|\tilde{u}_{n,n_0,\theta_j}(\theta_{j+1})\|^2 =
\sum_{\pi\in \Pi(n,n)} \tilde{C}_{\pi}^{\#},
\end{equation}
\begin{equation}
\tilde{C}_{\pi}^{\#}:=\int d\boldsymbol\xi d\boldsymbol\zeta
\tilde{K}^{\#}(\theta_{j+1},\theta_j;\bxi,n,n_0)
\overline{\tilde{K}^{\#}(\theta_{j+1},\theta_j;\bzeta,n,n_0)}
\Delta_{\pi} F(\bxi,n) \overline{F(\bzeta,n)},
\end{equation}
where
\begin{equation}
\tilde{K}^{\#}(\theta_{j+1},\theta_j;\bxi,n,n_0):=\int_{\theta_j}^{\theta_{j+1}}
K(\theta_{j+1}-s;\bxi,I_{1,n-n_0})K(s,\bxi,I_{n-n_0+1,n})ds.
\end{equation}
Recall Lemma 5.3. We can extend it to the following identity for
$K^{\#}(\theta_{j+1},\theta_j;\bxi,n,n_0)$:
\begin{equation}
\begin{aligned}
K^{\#}(\theta_{j+1},\theta_j;{\bxi},n,n_0)=&-\int_{\theta_j}^{\theta_{j+1}}ds e^{(\theta_{j+1}-s)\eta} e^{s\tilde{\eta}} \int_{-\infty}^{+\infty} d\alpha d\tilde{\alpha} e^{-i\alpha(\theta_{j+1}-s)} e^{-i\tilde{\alpha}s} \\
&\times\displaystyle\prod_{k=0}^{n-n_0}\frac{1}{\alpha+\xi_k^{\mathfrak{m}}+i\eta}\prod_{k\in
I_{n-n_0+1,n}}\frac{1}{\tilde{\alpha}+\xi_k^{\mathfrak{m}}+i\tilde{\eta}},
\end{aligned}
\end{equation}
where we choose $\eta:=(t/{\kappa})^{-1}$,$\tilde{\eta}:=t^{-1}$. We
can integrate  in $s$ to have
\begin{equation}
\begin{aligned}
K^{\#}(\theta_{j+1},\theta_j;{
\bxi},n,n_0)=&i\int_{-\infty}^{+\infty} d\alpha d{\tilde{\alpha}}
\frac{e^{-i\theta_{j+1}\tilde{\alpha}+\theta_{j+1}\tilde{\eta}}-e^{-i\alpha(\theta_{j+1}-\theta_j)-i\theta_j\tilde{\alpha}
+(\theta_{j+1}-\theta_j)\eta+\theta_{j}\tilde{\eta}}}{\alpha-\tilde{\alpha}+i(\eta-\tilde{\eta})} \\
&\times\displaystyle\prod_{k=0}^{n-n_0}\frac{1}{\alpha+\xi_k^{\mathfrak{m}}+i\eta}\prod_{k\in
I_{n-n_0+1,n}}\frac{1}{\tilde{\alpha}+\xi_k^{\mathfrak{m}}+i\tilde{\eta}}.
\end{aligned}
\end{equation}
Hence we can bound $C_\pi^{\#}$ by
\begin{equation}
\label{CP}
\begin{aligned}
|C_{\pi}^{\#}|\leq &\int d\xi \int_{-\infty}^{+\infty} d\alpha d\tilde{\alpha} d\beta d\tilde{\beta} \frac{1}{|\alpha-\tilde{\alpha}+i(\eta-\tilde{\eta})|} \frac{1}{|\beta-\tilde{\beta}+i(\eta-\tilde{\eta})|}\\
&\times\displaystyle\prod_{0\leq k \leq 2n+1}
\frac{1}{|\alpha_k+\xi_k^{\mathfrak{m}}+i\eta_k|}\prod_{k=1,k\neq {n+1}}^{2n+1} r(\varepsilon(\xi_k-\xi_{k-1})),
\end{aligned}
\end{equation}
where $\alpha_k$ and $\eta_k$ are defined as
\begin{equation}
\alpha_k :=\left\{
\begin{array}{llll}
\tilde{\alpha} & if\ k\leq n-n_0\\
\alpha & if\ n-n_0+1\leq k \leq n\\
\beta & if\ n+1\leq k \leq n+n_0\\
\tilde{\beta} & if\ k\geq n+n_0+2
\end{array}
\right.
\end{equation}
\begin{equation}
\eta_k :=\left\{
\begin{array}{llll}
\tilde{\eta} & if\ k\leq n-n_0\\
\eta & if\ n-n_0+1\leq k \leq n\\
\eta & if\ n+1\leq k \leq n+n_0\\
\tilde{\eta} & if\ k\geq n+n_0+2.
\end{array}
\right.
\end{equation}
We present the following lemmas to show $\displaystyle\lim_{\varepsilon\rightarrow 0}\mathbb{E}\|\Psi_{n_0,\varepsilon}\|_2^2\rightarrow0$.\\

\begin{lemma}
Let $n=4n_0$. For any $\pi\in \Pi(n,n)$ we have
\begin{equation}
\label{3.1} |\tilde{C}_{\pi}^{\#}|\leq
\frac{(C|\log{\varepsilon}|)^{4n_0}}{\kappa^{n_0}}.
\end{equation}
\end{lemma}
\begin{proof}
The following bound can be easily obtained by using Lemma
\begin{equation}
|\tilde{C}_{\pi}^{\#}|\leq (C|\log{\varepsilon}|)^{4n_0}.
\end{equation}
To recover the denominator in \eqref{3.1}, notice that among the
$\eta_k$ for $k\in B_0$, there are at least $n-2n_0-2(\geq n_0)$ of
them with $\eta_k=\kappa/t$. Hence
\begin{equation}
\prod_{k\in B_0} \frac{1}{|\alpha_k+\xi_k^{\mathfrak{m}}+i\eta_k|}
\leq \prod_{k \in B_0} |\eta_k|^{-1} \leq t^n \left(
\frac{t}{\kappa} \right)^{n_0} \leq \frac{t^n}{\kappa^{n_0}}.
\end{equation}
\end{proof}

\begin{lemma}
If $\pi\in\Pi_c(n,n)$, then we have
\begin{equation}
|C_{\pi}^{\#}|\leq
(C\log{\varepsilon})^n\varepsilon^{\lambda}\|\Phi(\xi)\|_2^2.
\end{equation}
\end{lemma}
\begin{proof}
The proof of Lemma 3.2 is essentially the same as in Lemma 2.1. The only
difference is that the integral in $\alpha$ and $\tilde{\alpha}$
($\beta$ and $\tilde{\beta}$) is estimated using Proposition 5.3 by
\begin{equation}
\begin{aligned}
\int\int
\frac{1}{|\tilde{\alpha}-\alpha+i(\eta-\tilde{\eta})|}\frac{1}{|\tilde{\alpha}+\xi_n^{\mathfrak{m}}+i\eta|}\
\frac{1}{|\alpha+\xi_0^{\mathfrak{m}}/{\varepsilon^{\mathfrak{m}}}+i\tilde{\eta}|}d\alpha
d\tilde{\alpha} \leq
\frac{C\left[1+\log_{+}\left|\frac{\xi_n^{\mathfrak{m}}-\xi_0^{\mathfrak{m}}/{\varepsilon}^{\mathfrak{m}}}{\tilde{\eta}}\right|\right]^2}
{|\xi_n^{\mathfrak{m}}-\xi_0^{\mathfrak{m}}/{\varepsilon^{\mathfrak{m}}}+i{\tilde{\eta}}|},
\end{aligned}
\end{equation}
and the integral in $\xi_0$ is estimated using Proposition 5.6
\begin{equation}
\varepsilon^{-\mathfrak{m}}\int
\frac{1}{\langle\xi_0\rangle^{2d}}\frac{\left[1+\log_{+}\left|\frac{\xi_0^{\mathfrak{m}}/{\varepsilon^{\mathfrak{m}}}-\xi_n^{\mathfrak{m}}}{\tilde{\eta}}\right|\right]^4}
{|\xi_0^{\mathfrak{m}}/{\varepsilon^{\mathfrak{m}}}-\xi_n^{\mathfrak{m}}+i\tilde{\eta}|^2}d\xi_0\leq
C\varepsilon^{\lambda}(\xi_n^{\lambda}\vee 1).
\end{equation}
\end{proof}

\begin{lemma}
If $\pi\in\Pi_{ncs}(n,n)$, then we have
\begin{equation}
|C_{\pi}^{\#}|\leq (C\log{\varepsilon})^n
\varepsilon^{2\lambda(1-\delta)}\|\Phi(\xi)\|_2^2.
\end{equation}
\end{lemma}
\begin{proof}
The proof of Lemma 3.3 is similar to that of Lemma 2.2. The only
difference is that the integral in $\alpha$ and $\tilde{\alpha}$
($\beta$ and $\tilde{\beta}$) is estimated by
\begin{equation}
\begin{aligned}
&\int\int
(\alpha^{\frac{\lambda(1-\delta)}{\mathfrak{m}}}\vee 1)
\frac{1}{|\tilde{\alpha}-\alpha+i(\eta-\tilde{\eta})|}\frac{1}{|\tilde{\alpha}+\xi_0^{\mathfrak{m}}+i\tilde{\eta}|}\frac{1}{|\alpha+\xi_0^{\mathfrak{m}}+i\eta|}d\alpha d\tilde{\alpha}\\
\leq &C
\int(\alpha^{\frac{\lambda(1-\delta)}{\mathfrak{m}}}\vee 1)\frac{1}{|{\alpha}+\xi_0^{\mathfrak{m}}+i\tilde{\eta}|^2}
\left[\log_+\left|\frac{{\alpha}+\xi_0^{\mathfrak{m}}}{\tilde{\eta}}\right|+1\right]d{\alpha}\leq C\xi_0^{\lambda(1-\delta)}.
\end{aligned}
\end{equation}
\end{proof}

\begin{lemma}
If $\pi\in\Pi_s(n,n)$, then we have
\begin{equation}
|C_{\pi}^{\#}|\leq \frac{C^n}{(n!)^{1/2}}\|\hat{u}_0(\xi)\|_2^2.
\end{equation}
\end{lemma}
\begin{proof}
The proof is essentially the same as in Lemma 2.3. We shall not
repeat the argument here.
\end{proof}
\noindent We now apply the above lemmas to estimate
$\Psi_{n_0,\varepsilon}$. From Lemma 3.2, 3.3, and 3.4 we have
\begin{equation}
\|U_1\|_2^2\leq
n_0^2\kappa^2[(C\log\varepsilon)^{4n_0}\varepsilon^{\lambda}](\frac{8n_0}{e})^{4n_0}+\frac{C^{n_0}n_0^2\kappa^2}{(n_0!)^{1/2}}.
\end{equation}
From Lemma 3.1 we have
\begin{equation}
\|U_2\|_2^2\leq
\frac{(C|\log\varepsilon|)^{4n_0}}{\kappa^{n_0}}(\frac{8n_0}{e})^{4n_0}.
\end{equation}
The $L^2$ estimate of $\Psi_{n_0,\varepsilon}$ is therefore given by
\begin{equation}
\label{uErr}
\mathbb{E}\|\sum_{n_0}^{\infty}\hat{u}_{\varepsilon,n}(t)\|_2^2=\mathbb{E}\|\Psi_{n_0,\varepsilon}\|_2^2\leq2(\mathbb{E}\|U_1\|_2^2+\mathbb{E}\|U_2\|_2^2)\rightarrow
0
\end{equation}
as $\varepsilon\rightarrow 0$.

\section{Homogenization}
We come back to the analysis of $U_{\varepsilon,s}(t,\xi)$. We find
that $U_{\varepsilon,s}$ is the solution to the following equation
\begin{equation}
\begin{aligned}
\label{Usimp}
U_{\varepsilon,s}&=e^{it{\xi}^{\mathfrak{m}}}\hat{u}_0(\xi)-\int_0^t e^{is{\xi}^{\mathfrak{m}}}\int_0^{t-s} e^{i s_1 \xi_1^{\mathfrak{m}}} \int \varepsilon^{d-\mathfrak{m}}\hat{R}(\varepsilon(\xi_1-\xi)) U_{\varepsilon,s} (t-s-s_1,\xi) d\xi_1 ds ds_1\\
&:=e^{it{\xi}^{\mathfrak{m}}}\hat{u}_0(\xi)+A_{\varepsilon}U_{\varepsilon,s}(t,\xi).
\end{aligned}
\end{equation}
\begin{lemma}
Let us define $\mathfrak{U}_{\epsilon}(t,\xi)$ to be the solution to
\begin{equation}
\label{frakU}
\begin{aligned}
(i\frac{\partial}{\partial t}+\xi^{\mathfrak{m}}-\rho_{\epsilon}(\xi))\mathfrak{U}_{\epsilon}(t,\xi)&=0\\
\mathfrak{U}_{\epsilon}(0,\xi)&=\hat{u}_0(\xi),
\end{aligned}
\end{equation}
with
$\rho_\epsilon=\int_{\mathbb{R}^d}\frac{\hat{R}(\xi_1-\epsilon\xi)}{\xi_1^{\mathfrak{m}}}d\xi_1$.
We have the convergence results
\begin{equation}
|(U_{\varepsilon,s}-\mathfrak{U}_{\varepsilon})(t)|\lesssim
\max\{\varepsilon^{\mathfrak{m}},\varepsilon^{d-\mathfrak{m}}|\log\varepsilon|\}.
\end{equation}
\end{lemma}
\begin{proof}
(1) We obtain from Duhamel's principle that
\begin{equation}
\begin{aligned}
\mathfrak{U}_{\varepsilon}(t,\xi)&=
e^{it{\xi}^{\mathfrak{m}}}\hat{u}_0(\xi)-i\int_0^t
e^{i\xi^{\mathfrak{m}}v}\int_{\mathbb{R}^d}
\frac{\hat{R}(\xi_1-\varepsilon\xi)}{\xi_1^{\mathfrak{m}}}d\xi_1
\mathfrak{U}_{\varepsilon}(t-v,\xi)dv\\
&:=e^{it{\xi}^{\mathfrak{m}}}\hat{u}_0(\xi)+B_{\varepsilon} \mathfrak{U}_{\varepsilon}(t,\xi)\\
&:=e^{it{\xi}^{\mathfrak{m}}}\hat{u}_0(\xi)+A_{\varepsilon}
\mathfrak{U}_{\varepsilon}(t,\xi)+E_{\varepsilon}
\mathfrak{U}_{\varepsilon}(t,\xi),
\end{aligned}
\end{equation}
where the operator $A_{\varepsilon}$ is defined in \eqref{Usimp} and may be recast as
\begin{equation}
\begin{aligned}
A_{\varepsilon} \mathfrak{U}_{\varepsilon} &= -\int_0^t e^{i\xi^{\mathfrak{m}} v}\int_{\mathbb{R}^d} \int_0^{\frac{v}{\varepsilon^{\mathfrak{m}}}} e^{is_1(\xi_1^{\mathfrak{m}}-\varepsilon^{\mathfrak{m}}\xi^{\mathfrak{m}})}ds_1 \hat{R}(\xi_1-\varepsilon\xi)d\xi_1 \mathfrak{U}_{\varepsilon}(t-v,\xi)dv\\
&= -\int_0^t e^{i\xi^{\mathfrak{m}}v} \int_{\mathbb{R}^d}
\frac{1}{i(\xi_1^{\mathfrak{m}}-\varepsilon^{\mathfrak{m}}\xi^{\mathfrak{m}})}
(e^{i
\frac{v}{\varepsilon^{\mathfrak{m}}}(\xi_1^{\mathfrak{m}}-\varepsilon^{\mathfrak{m}}\xi^{\mathfrak{m}})}-1)\hat{R}
(\xi_1-\varepsilon\xi)d\xi_1 \mathfrak{U}_{\varepsilon}(t-v,\xi)dv,
\end{aligned}
\end{equation}
and the remainder $E_{\varepsilon}$ is then given by
\begin{equation}
\begin{aligned}
E_{\varepsilon}\mathfrak{U}_{\varepsilon}&=i\int_0^t
\int_{\mathbb{R}^d}\frac{1}{\xi_1^{\mathfrak{m}}-\varepsilon^{\mathfrak{m}}\xi^{\mathfrak{m}}} {e^{i
\frac{v}{\varepsilon^{\mathfrak{m}}}\xi_1^{\mathfrak{m}}}}
\hat{R}(\xi_1-\varepsilon\xi)d\xi_1 \mathfrak{U}_{\varepsilon}(t-v,\xi)dv\\
&+\ i\int_0^t\int_{\mathbb{R}^d}
\frac{\varepsilon^{\mathfrak{m}}\xi^{\mathfrak{m}}}{\xi_1^\mathfrak{m}(\xi_1^{\mathfrak{m}}-\varepsilon^{\mathfrak{m}}
\xi^{\mathfrak{m}})}(e^{i\xi^{\mathfrak{m}}v}-e^{i\frac{\xi_1^{\mathfrak{m}}}{{\varepsilon^{\mathfrak{m}}}}v})\hat{R}
(\xi_1-\varepsilon\xi)d\xi_1 \mathfrak{U}_{\varepsilon}(t-v,\varepsilon)dv\\
&:=I_1+I_2.
\end{aligned}
\end{equation}
For the calculation of $I_1$, note that equation \eqref{frakU} has the explicit solution:
\begin{equation}
\mathfrak{U}(t,\xi)=e^{it(\xi^{\mathfrak{m}}-\rho_{\varepsilon}(\xi))}\hat{u}_0(\xi).
\end{equation}
On one hand, we may obtain the following expression of integral in $v$ using the method of separation of variables:
\begin{equation}
\begin{aligned}
\int_0^t e^{i\frac{v}{\varepsilon^{\mathfrak{m}}}(\xi_1^{\mathfrak{m}}-\varepsilon^{\mathfrak{m}}\xi^{\mathfrak{m}})}
\mathfrak{U}_{\varepsilon}(t-v,\xi)dv
=&\frac{\varepsilon^{\mathfrak{m}}}{\xi_1^{\mathfrak{m}}-\varepsilon^{\mathfrak{m}}\xi^{\mathfrak{m}}}\int_0^t \mathfrak{U}_{\varepsilon}(t-v,\xi)d(e^{i\frac{v}{\varepsilon^{\mathfrak{m}}}(\xi_1^{\mathfrak{m}}-\varepsilon^{\mathfrak{m}}
\xi^{\mathfrak{m}})})\\
=&\frac{\varepsilon^{\mathfrak{m}}}{\xi_1^{\mathfrak{m}}-\varepsilon^{\mathfrak{m}}\xi^{\mathfrak{m}}}
[e^{i\frac{t}{\varepsilon^{\mathfrak{m}}}(\xi_1^{\mathfrak{m}}-\varepsilon^{\mathfrak{m}}\xi^{\mathfrak{m}})}\hat{u}_0(\xi)
-e^{i(t-v)(\xi^{\mathfrak{m}}-\rho_{\varepsilon}(\xi))}\hat{u}_0(\xi)\\
&+i\int_0^t e^{i\frac{v}{\varepsilon^{\mathfrak{m}}}(\xi_1^{\mathfrak{m}}-\varepsilon^{\mathfrak{m}}\xi^{\mathfrak{m}})}
e^{i(t-v)(\xi^{\mathfrak{m}}-\rho_{\varepsilon}(\varepsilon))}\hat{u}_0(\xi)dv].
\end{aligned}
\end{equation}
On the other hand, we have the following simple estimate of this integral:
\begin{equation}
\left|\int_0^t  e^{i\frac{v}{\varepsilon^{\mathfrak{m}}}(\xi_1^{\mathfrak{m}}-\varepsilon^{\mathfrak{m}}\xi^{\mathfrak{m}})}
\mathfrak{U}_{\varepsilon}(t-v,\xi)dv\right|
\leq \int_0^t |e^{i\frac{v}{\varepsilon^{\mathfrak{m}}}(\xi_1^{\mathfrak{m}}-\varepsilon^{\mathfrak{m}}\xi^{\mathfrak{m}})}
\mathfrak{U}_{\varepsilon}(t-v,\xi)|dv
\leq C|\hat{u}_0(\xi)|.
\end{equation}
This gives that
\begin{equation}
\left|\int_0^t e^{i\frac{v}{\varepsilon^{\mathfrak{m}}}(\xi_1^{\mathfrak{m}}-\varepsilon^{\mathfrak{m}}\xi^{\mathfrak{m}})}
\mathfrak{U}_{\varepsilon}(t-v,\xi)dv\right|
\leq C\varepsilon^{\mathfrak{m}} (\frac{1}{\xi_1^{\mathfrak{m}}-\varepsilon^{\mathfrak{m}}\xi^{\mathfrak{m}}}\wedge \frac{1}{\varepsilon^{\mathfrak{m}}})\max\{|\hat{u}_0(\xi)|,|\xi^{\mathfrak{m}}\hat{u}_0(\xi)|\}.
\end{equation}
We therefore find the estimate of $|I_1|$ using Lemma 5.9:
\begin{equation}
\begin{aligned}
|I_1|&\leq C\varepsilon^{\mathfrak{m}}\int_{\mathbb{R}^d} (\frac{1}{\xi_1^{\mathfrak{m}}-\varepsilon^{\mathfrak{m}}\xi^{\mathfrak{m}}}\wedge\frac{1}{\varepsilon^{\mathfrak{m}}})
\max\{|\hat{u}_0(\xi)|,|\xi^{\mathfrak{m}}\hat{u}_0(\xi)|\}\hat{R}(\xi_1-\varepsilon\xi)d\xi_1\\
&\leq C\max\{1,\xi^{\mathfrak{m}}\}\max\{\varepsilon^{\mathfrak{m}},\varepsilon^{d-\mathfrak{m}}|\log\varepsilon|(\xi^{\mathfrak{m}}+1)^{\frac{d}{\mathfrak{m}}-2},
\varepsilon^{d-\mathfrak{m}}(\xi^{\mathfrak{m}}+1)^{\frac{d}{\mathfrak{m}}-2}\}|\hat{u}_0(\xi)|.
\end{aligned}
\end{equation}
For the calculation of $I_2$, we first estimate the integral in $v$:
\begin{equation}
\varepsilon^{\mathfrak{m}}\left|\int_0^t \frac{e^{i\xi^{\mathfrak{m}}v}-e^{i\frac{\xi_1^{\mathfrak{m}}}{\varepsilon^{\mathfrak{m}}}v}}
{\xi_1^{\mathfrak{m}}-\varepsilon^{\mathfrak{m}}\xi^{\mathfrak{m}}}\mathfrak{U}_{\varepsilon}(t-v,\xi)dv\right|
\leq C(\frac{1}{\xi_1^{\mathfrak{m}}-\varepsilon^{\mathfrak{m}}\xi^{\mathfrak{m}}}\wedge\frac{1}{\varepsilon^{\mathfrak{m}}})
\max\{|\hat{u}_0(\xi)|,|\xi^{\mathfrak{m}}\hat{u}_0(\xi)|\}.
\end{equation}
Hence
\begin{equation}
\begin{aligned}
|I_2|&\leq
C\int_{\mathbb{R}^d} \frac{\xi^{\mathfrak{m}}}{\xi_1^{\mathfrak{m}}}(\frac{1}{\xi_1^{\mathfrak{m}}-\varepsilon^{\mathfrak{m}}\xi^{\mathfrak{m}}}\wedge\frac{1}{\varepsilon^{\mathfrak{m}}})
\hat{R}(\xi_1-\varepsilon\xi)\max\{|\hat{u}_0(\xi)|,|\xi^{\mathfrak{m}}\hat{u}_0(\xi)|\}d\xi_1\\
&\leq C\max\{1,\xi^{\mathfrak{m}}\}\max\{\varepsilon^{\mathfrak{m}},\varepsilon^{d-\mathfrak{m}}|\log\varepsilon|(\xi^{\mathfrak{m}}+1)^{\frac{d}{\mathfrak{m}}-2},
\varepsilon^{d-\mathfrak{m}}(\xi^{\mathfrak{m}}+1)^{\frac{d}{\mathfrak{m}}-2}\}|\xi^{\mathfrak{m}}\hat{u}_0(\xi)|.
\end{aligned}
\end{equation}

\noindent Note that
\begin{equation}
|A_{\varepsilon}U(t,\xi)| \leq C \int_0^t |U(s,\xi)| ds,
\end{equation}
over a bounded interval in time. The equation
\begin{equation}
(I-A_{\varepsilon})U(t,\xi)=S(t,\xi)
\end{equation}
therefore admits a unique solution by Gronwall's Lemma, which is bounded by
\begin{equation}
|U(t,\xi)|\leq \|S\|_{\infty} e^{Ct}.
\end{equation}
We verify that the solution to
\begin{equation}
(I-B_{\varepsilon})\mathfrak{U}_{\varepsilon} =
e^{it\xi^{\mathfrak{m}}} \hat{u}_0(\xi),
\end{equation}
is given by
\begin{equation}
\mathfrak{U}_{\varepsilon}(t,\xi)=e^{it(\xi^{\mathfrak{m}}-\rho_{\varepsilon}(\xi))}\hat{u}_0(\xi).
\end{equation}
The error
$V_{\varepsilon}(t,\xi)=(U_{\varepsilon,s}(t,\xi)-\mathfrak{U}_{\varepsilon}(t,\xi))$
is a solution to
\begin{equation}
(I-A_{\varepsilon})V_{\varepsilon} =
E_{\varepsilon}\mathfrak{U}_{\varepsilon}(t,\xi),
\end{equation}
so that over bounded intervals in time, we find that
\begin{equation}
|U_{\varepsilon,s}(t,\xi)-\mathfrak{U}_{\varepsilon}(t,\xi)| =
|V_{\varepsilon}(t,\xi)| \lesssim \max\{\varepsilon^{\mathfrak{m}},\varepsilon^{d-\mathfrak{m}}|\log\varepsilon|\}.
\end{equation}
\end{proof}

From our assumption that
$\hat{R}(\xi)\in\mathcal{C}^2(\mathbb{R}^d)$, we find that
\begin{equation}
|e^{it(\xi^{\mathfrak{m}}-\rho_{\varepsilon}(\xi))}-e^{it(\xi^{\mathfrak{m}}-\rho)}|
\leq t|\rho_{\varepsilon}(\xi)-\rho)| \leq C\varepsilon^2\xi^2.
\end{equation}
The reason for the second-order accuracy is that
$\hat{R}(-\xi)=\hat{R}(\xi)$ and $\nabla \hat{R}(0)=0$ so that
first-order terms in the Taylor expansion vanish.

\noindent In terms of the solutions of PDE we defined in \eqref{uhom}, we may recast the above result as
\begin{equation}
\label{est4.1} |(U_{\varepsilon,s}-\mathcal{U})(t)|
\lesssim \max\{\varepsilon^{2},\varepsilon^{\mathfrak{m}},\varepsilon^{d-\mathfrak{m}}|\log\varepsilon|\}
\qquad \mathcal{U}(t,\xi)=e^{i(\xi^{\mathcal{m}}-\rho)t}\hat{u}_0(\xi),
\end{equation}
\noindent We now prove Theorem \ref{thm}. By the triangle inequality, we
have the estimate
\begin{equation}
\begin{aligned}
&\mathbb{E}\|\hat{u}_{\varepsilon}(t,\xi)-{\mathcal{U}}(t,\xi)\|_2\\
\leq
&\mathbb{E}\|\sum_{n=0}^{n_0(\varepsilon)-1}(u_{\varepsilon,n}-U_{\varepsilon,s}^n)(t)\|_2+\mathbb{E}\|\sum_{n_0(\varepsilon)}^{+\infty}\hat{u}_{\varepsilon,n}(t)\|_2
+\|(U_{\varepsilon,s}-{\mathcal{U}})(t)\|_2+\|\sum_{n_0(\varepsilon)}^{+\infty}U_{\varepsilon,s}^n(t)\|_2.
\end{aligned}
\end{equation}
The vanishing of the first three terms on right hand side of this inequality when $\varepsilon$ goes to zero follows from \eqref{est2.2.2}, \eqref{uErr} and \eqref{est4.1} respectively. The fourth term also vanishes because of the $L^2$ convergence of $U_{\varepsilon,s}$.
\section{Inequalities and Proofs}
In this section, we present and prove several inequalities used in earlier sections. There are similar versions of
Lemma 5.2 and 5.7 in \cite{Erdos-Yau2}. The proofs are given below for the convenience of the reader.
The proofs of similar versions of Lemma 5.1 and 5.8 can be found in
\cite{Erdos-Yau2}.
\begin{lemma}
We have the following identity for $\eta>0$:
\begin{equation}
\label{5.3.1} K(t,\bxi,I)=ie^{t\eta}\int d\alpha e^{-i\alpha
t}\prod_{k\in I} \frac{1}{\alpha+\xi_k^{\mathfrak{m}}+i\eta}.
\end{equation}
We also claim the following estimate with $n:=|I|-1$:
\begin{equation}
\label{5.3.2} |K(t,\bxi,I)|\leq \frac{t^n}{n!}.
\end{equation}
\end{lemma}

\begin{lemma}
Assume $\eta>0$. We have the following
inequality:\\
\begin{equation}
\label{5.2.1} \int_{-\infty}^{\infty}
\frac{d\alpha}{|\alpha+A+i\eta||\alpha+B+i\eta|}\leq
\frac{C}{|A-B+i\eta|}\left[1+\log_{+}\left|\frac{A-B}{\eta}\right|\right],
\end{equation}
where $\log_{+}x:=\max\{0,\log x\}$ for $x>0$ and $\log_{+}
0:=0$.\\
\end{lemma}
\begin{proof}
Without loss of generality, we assume $B>A$. We split the
integration over $(-\infty,-\frac{A+B}{2})$ as follows:
\begin{equation}
\label{split}
\begin{aligned}
&\int_{-\infty}^{-\frac{A+B}{2}}\frac{d\alpha}{|\alpha+A+i\eta||\alpha+B+i\eta|}\\
=
&\int_{-B}^{-\frac{A+B}{2}}\frac{d\alpha}{|\alpha+A+i\eta||\alpha+B+i\eta|}
+
\int_{-B-(B-A)}^{-B}\frac{d\alpha}{|\alpha+A+i\eta||\alpha+B+i\eta|}\\
+&\int_{-\infty}^{-B-(B-A)}\frac{d\alpha}{|\alpha+A+i\eta||\alpha+B+i\eta|}.
\end{aligned}
\end{equation}
The first term is estimated as
\begin{equation}
\begin{aligned}
\int_{-B}^{-\frac{A+B}{2}}\frac{d\alpha}{|\alpha+A+i\eta||\alpha+B+i\eta|}
\leq
&\frac{1}{|-\frac{A+B}{2}+A+i\eta|}\int_{-B}^{-\frac{A+B}{2}}\frac{d\alpha}{|\alpha+B+i\eta|}\\
\leq &\frac{2}{|A-B+i\eta|}\int_0^{\frac{B-A}{2\eta}}\frac{d\alpha}{\sqrt{\alpha^2+1}}\\
\leq
&\frac{C}{|A-B+i\eta|}\left[1+\log_{+}\left|\frac{A-B}{\eta}\right|\right].
\end{aligned}
\end{equation}
Likewise, the second term is estimated as
\begin{equation}
\begin{aligned}
\int_{-B-(B-A)}^{-B}\frac{d\alpha}{|\alpha+A+i\eta||\alpha+B+i\eta|}
\leq
&\frac{1}{|A-B+i\eta|}\int_{-B-(B-A)}^{-B}\frac{d\alpha}{|\alpha+B+i\eta|}\\
\leq &\frac{1}{|A-B+i\eta|}\int_0^{\frac{B-A}{\eta}}\frac{d\alpha}{\sqrt{\alpha^2+1}}\\
\leq
&\frac{C}{|A-B+i\eta|}\left[1+\log_{+}\left|\frac{A-B}{\eta}\right|\right].
\end{aligned}
\end{equation}
We obtain the bound for the third term by using the inequality
$|\alpha+A+i\eta|\leq |\alpha+B+i\eta|$ on $(-\infty,-B-(B-A))$.
\begin{equation}
\begin{aligned}
\int_{-\infty}^{-B-(B-A)}\frac{d\alpha}{|\alpha+A+i\eta||\alpha+B+i\eta|}
&\leq \int_{-\infty}^{-B-(B-A)}\frac{d\alpha}{|\alpha+B+i\eta|^2}\\
&=\frac{1}{\eta}\int_{\frac{B-A}{\eta}}^{+\infty}\frac{d\alpha}{\alpha^2+1}.
\end{aligned}
\end{equation}
If $B-A\geq \eta$, we have
\begin{equation}
\label{5.1.beg}
\begin{aligned}
\int_{\frac{B-A}{\eta}}^{+\infty}\frac{d\alpha}{\alpha^2+1}\leq
\int_{\frac{B-A}{\eta}}^{+\infty}\frac{d\alpha}{\alpha^2}\leq
\frac{\eta}{B-A},
\end{aligned}
\end{equation}
in which case
\begin{equation}
\int_{-\infty}^{-B-(B-A)}\frac{d\alpha}{|\alpha+A+i\eta||\alpha+B+i\eta|}\leq
\frac{1}{B-A}\leq \frac{\sqrt{2}}{|A-B+i\eta|}.
\end{equation}
If $B-A<\eta$, we have
\begin{equation}
\begin{aligned}
\int_{\frac{B-A}{\eta}}^{+\infty}\frac{d\alpha}{\alpha^2+1}\leq
\int_1^{+\infty} \frac{d\alpha}{\alpha^2+1}
+\int_{\frac{B-A}{\eta}}^1 \frac{d\alpha}{\alpha^2+1}\leq
\frac{\pi}{4}+(1-\frac{B-A}{\eta}),
\end{aligned}
\end{equation}
in which case
\begin{equation}
\label{5.1.end}
\int_{-\infty}^{-B-(B-A)}\frac{d\alpha}{|\alpha+A+i\eta||\alpha+B+i\eta|}\leq
\frac{1}{\eta}(\frac{\pi}{4}+1)\leq
\frac{\sqrt{2}({\pi}/{4}+1)}{|A-B+i\eta|}.
\end{equation}
By symmetry, the integration over $(-\infty,-\frac{A+B}{2})$ admits
an identical bound.
\end{proof}

\begin{proposition}
Assume $\eta,\tilde{\eta}>0$, $\tilde{\eta}>2\eta$, and $\eta^{-1}$
bounded. We have the following inequality:
\begin{equation}
\label{5.2.2} \int\int
\frac{1}{|\tilde{\alpha}-\alpha+i(\eta-\tilde{\eta})|}\frac{1}{|\tilde{\alpha}+A+i\eta|}\
\frac{1}{|\alpha+B+i\tilde{\eta}|}d\alpha d\tilde{\alpha} \leq
\frac{C\left[1+\log_{+}\left|\frac{A-B}{\tilde{\eta}}\right|\right]^2}{|A-B+i\tilde{\eta}|}.
\end{equation}
\end{proposition}
\begin{proof}
Without loss of generality, we assume $B>A$.\\
We first integrate in $\tilde{\alpha}$ by using Lemma 5.2:
\begin{equation}
\int\frac{1}{|\tilde{\alpha}-\alpha+i(\eta-\tilde{\eta})|}\frac{1}{|\tilde{\alpha}+A+i{\eta}|}d\tilde{\alpha}
\leq \frac{1}{|\alpha+A+i\tilde{\eta}|}\left[1+\log_{+}\left|\frac{\alpha+A}{\tilde{\eta}}\right|\right].
\end{equation}
The integration over $(-\infty,-\frac{A+B}{2})$ is split as in \eqref{split} into three pieces: $(-B,-\frac{A+B}{2})$, $(-B-(B-A),-B)$ and $(-\infty,-B-(B-A))$.\\
The term $1+\log_{+}\left|\frac{\alpha+A}{\tilde{\eta}}\right|$ is bounded
by $1+\log_{+}\left|\frac{A-B}{\tilde{\eta}}\right|$ on the interval
$(-B,-\frac{A+B}{2})$. The integral over $(-B,-\frac{A+B}{2})$ is
then bounded by
\begin{equation}
\int_{-B}^{-\frac{A+B}{2}}
\frac{1+\log_{+}\left|\frac{\alpha+A}{\tilde{\eta}}\right|}{|\alpha+A+i\tilde{\eta}||\alpha+B+i\tilde{\eta}|}d\alpha
\leq
\frac{C}{|A-B+i\tilde{\eta}|}\left[1+\log_{+}\left|\frac{A-B}{\tilde{\eta}}\right|\right]^2.
\end{equation}
The same bound holds for the integral over $(-B-(B-A),-B)$.\\
It remains to estimate the integral over $(-\infty,-B-(B-A))$. The domain of integration can be divided into two parts. On the set
$|\alpha+A|<\tilde{\eta}$, we have
$\log_{+}\left|\frac{\alpha+A}{\tilde{\eta}}\right|=0$, so that applying Lemma
5.2 immediately gives the desired result. On the set
$|\alpha+A|\geq\tilde{\eta}$, we use the simple bound on this interval
\begin{equation}
1+\log_{+}\left|\frac{\alpha+A}{\tilde{\eta}}\right|\leq
\left|\frac{\alpha+A}{\tilde{\eta}}+i\right|^{1/2}.
\end{equation}
It then follows that
\begin{equation}
\begin{aligned}
\int_{-\infty}^{-B-(B-A)}
\frac{1+\log_{+}\left|\frac{\alpha+A}{\tilde{\eta}}\right|}{|\alpha+A+i\tilde{\eta}||\alpha+B+i\tilde{\eta}|}d\alpha
&\leq
C\int_{-\infty}^{-B-(B-A)}\frac{1}{|\alpha+B+i\tilde{\eta}|^{3/2}}d\alpha\\
&=C\int_{\frac{B-A}{\tilde{\eta}}}^{+\infty}
\frac{d\alpha}{(\alpha^2+1)^{3/2}}.
\end{aligned}
\end{equation}
After performing an analysis similar to that from
\eqref{5.1.beg} to \eqref{5.1.end} we find the estimate:
\begin{equation}
\int_{-\infty}^{-B-(B-A)}
\frac{1+\log_{+}\left|\frac{\alpha+A}{\tilde{\eta}}\right|}{|\alpha+A+i\tilde{\eta}||\alpha+B+i\tilde{\eta}|}d\alpha
\leq
\frac{C}{|A-B+i\tilde{\eta}|}\left[1+\log_{+}\left|\frac{A-B}{\tilde{\eta}}\right|\right],
\end{equation}
which concludes the proof.
\end{proof}

\begin{lemma}
Assume $\eta>0$ and $\eta^{-1}$ is bounded. We have the following inequality:
\begin{equation}
\label{5.1} \int \varepsilon^{\mathfrak{m}}
\frac{1}{\langle\xi\rangle^{2d}}
\frac{1}{|\varepsilon^{\mathfrak{m}}\alpha+\xi^{\mathfrak{m}}+i\varepsilon^{\mathfrak{m}}\eta|^2}d\xi
\leq C\varepsilon^{\lambda}(\alpha^{\frac{\lambda}{\mathfrak{m}}}\vee
1).
\end{equation}
\end{lemma}
\begin{proof}
The domain of integration may be split into two parts. On the set
$|\varepsilon^{\mathfrak{m}}\alpha+\xi^{\mathfrak{m}}|\geq 1$, the
integral is bounded by $\varepsilon^{\mathfrak{m}}
\|1/{\langle\xi\rangle^{2d}}\|_1$. On the set
$|\varepsilon^{\mathfrak{m}}\alpha+\xi^{\mathfrak{m}}|< 1$, we
change to polar coordinates
\begin{equation}
\int_{|\varepsilon^{\mathfrak{m}}\alpha+\xi^{\mathfrak{m}}|<1}
\varepsilon^{\mathfrak{m}} \frac{1}{\langle \xi \rangle^{2d}}
\frac{1}{|\epsilon^{\mathfrak{m}}\alpha+\xi^{\mathfrak{m}}+i\varepsilon^{\mathfrak{m}}\eta|^2}
d\xi \leq C
\int_{|\varepsilon^{\mathfrak{m}}\alpha+\xi^{\mathfrak{m}}|<1}
\varepsilon^{\mathfrak{m}} \frac{1}{\langle \xi \rangle^{2d}}
\frac{\xi^{d-1}}{|\epsilon^{\mathfrak{m}}\alpha+\xi^{\mathfrak{m}}+i\varepsilon^{\mathfrak{m}}\eta|^2}
d|\xi|.
\end{equation}
When $\mathfrak{m}<d<2\mathfrak{m}$, we use the inequality
$|1/{\langle\xi\rangle^{2d}}|\leq 1$ and change variables to
$Q=\varepsilon^{\mathfrak{m}}\alpha+\xi^{\mathfrak{m}}$
\begin{equation}
\int_{|\varepsilon^{\mathfrak{m}}\alpha+\xi^{\mathfrak{m}}|<1}
\varepsilon^{\mathfrak{m}}
\frac{\xi^{d-1}}{|\epsilon^{\mathfrak{m}}\alpha+\xi^{\mathfrak{m}}+i\varepsilon^{\mathfrak{m}}\eta|^2}d|\xi|
= \int_0^1 \varepsilon^{\mathfrak{m}}
\frac{(Q-\varepsilon^{\mathfrak{m}}\alpha)^{\frac{d}{\mathfrak{m}}-1}}
{|Q+i\varepsilon^{\mathfrak{m}}\eta|^2}dQ.
\end{equation}
On the set $[0,\varepsilon^{\mathfrak{m}}\eta]$, we bound
$|Q+i\varepsilon^{\mathfrak{m}}\eta|$ from below by its imaginary
part and estimate the integral
\begin{equation}
\int_0^{\varepsilon^{\mathfrak{m}}\eta} \varepsilon^{\mathfrak{m}}
\frac{(Q-\varepsilon^{\mathfrak{m}}\alpha)^{\frac{d}{\mathfrak{m}}-1}}{|\varepsilon^{\mathfrak{m}}\eta|^2}dQ
\leq  \int_0^{\varepsilon^{\mathfrak{m}}\eta}
\varepsilon^{\mathfrak{m}}
\frac{Q^{\frac{d}{\mathfrak{m}}-1}+(\varepsilon^{\mathfrak{m}}\alpha)^{\frac{d}{\mathfrak{m}}-1}}{|\varepsilon^{\mathfrak{m}}\eta|^2}dQ\leq
C\varepsilon^{d-\mathfrak{m}}(\alpha\vee 1)^{\frac{d}{\mathfrak{m}}-1}.
\end{equation}
On the set $[\varepsilon^{\mathfrak{m}}\eta,1]$, we bound
$|Q+i\varepsilon^{\mathfrak{m}}\eta|$ from below by its real part
and estimate the integral
\begin{equation}
\int_0^{\varepsilon^{\mathfrak{m}}\eta} \varepsilon^{\mathfrak{m}}
\frac{(Q-\varepsilon^{\mathfrak{m}}\alpha)^{\frac{d}{\mathfrak{m}}-1}}{Q^2}dQ
\leq  \int_0^{\varepsilon^{\mathfrak{m}}\eta}
\varepsilon^{\mathfrak{m}}
\frac{Q^{\frac{d}{\mathfrak{m}}-1}+(\varepsilon^{\mathfrak{m}}\alpha)^{\frac{d}{\mathfrak{m}}-1}}{Q^2}dQ\leq
C\varepsilon^{d-\mathfrak{m}}(\alpha\vee 1)^{\frac{d}{\mathfrak{m}}-1}.
\end{equation}
This concludes the proof for $\mathfrak{m}<d<2\mathfrak{m}$.
When $d\geq2\mathfrak{m}$, we use the inequality
$|\xi^{d-2\mathfrak{m}}/\langle \xi \rangle^{2d}|\leq 1$ instead of
$|1/{\langle\xi\rangle^{2d}}|\leq 1$ and the rest of the proof is
similar.
\end{proof}

\begin{proposition}
Assume $\eta>0$ and $\eta^{-1}$ is bounded. We have the following inequality:
\begin{equation}
\label{5.1.1} \int \varepsilon^{\mathfrak{m}}
\frac{1}{\langle\xi\rangle^{2d}}
\frac{1}{|\varepsilon^{\mathfrak{m}}\alpha+\xi^{\mathfrak{m}}+i\varepsilon^{\mathfrak{m}}\eta|^2}d\xi
\leq C\varepsilon^{\lambda(1-\delta)}(\alpha^{\frac{\lambda(1-\delta)}{\mathfrak{m}}}\vee
1).
\end{equation}
\end{proposition}
\begin{proof}
To prove \eqref{5.1.1}, we just need to use ${1}/{\langle\xi\rangle^{2d}}$ to bound an additional $\xi^{(d-\mathfrak{m})\delta}$ and the rest of the proof is the same as in Lemma 5.4. As a result, we lose
a contribution $\varepsilon^{(d-\mathfrak{m})\delta}$ for lowering the order of $\alpha$ by $\frac{\beta\delta}{\mathfrak{m}}$.
\end{proof}

\begin{proposition}
Assume $\eta>0$, $\eta^{-1}$ bounded, and $k\in\mathbb{N}$. We have
the following inequality:
\begin{equation}
\label{5.2.3} \int \varepsilon^{\mathfrak{m}}
\frac{1}{\langle\xi\rangle^{2d}}
\frac{\left[1+\log_{+}\left|\frac{\alpha+\xi^{\mathfrak{m}}/{\varepsilon^{\mathfrak{m}}}}{\eta}\right|\right]^k}
{|\varepsilon^{\mathfrak{m}}\alpha+\xi^{\mathfrak{m}}+i\varepsilon^{\mathfrak{m}}\eta|^2}d\xi
\leq C\varepsilon^{\lambda}(\alpha\vee
1)^{\frac{\lambda}{\mathfrak{m}}}.
\end{equation}
\end{proposition}
\begin{proof}
The integral on the set $\left|\frac{\alpha+\xi^{\mathfrak{m}}/{\varepsilon^{\mathfrak{m}}}}{\eta}\right|< 1$ is immediately bounded by $C\varepsilon^{\lambda}\alpha^{\frac{\lambda}{\mathfrak{m}}}$ using Lemma 5.4. The integral on the set
$\left|\frac{\alpha+\xi^{\mathfrak{m}}/{\varepsilon^{\mathfrak{m}}}}{\eta}\right|\geq
1$ is bounded by
$C\varepsilon^{\mathfrak{m}}\|1/\langle\xi\rangle^{2d}\|_1$.
\end{proof}

\begin{lemma}
Assume $\eta>0$. We have the following
inequalities:\\
\begin{equation}
\label{5.2.4} \sup_{\omega} \int
\frac{1}{|\alpha+\xi^{\mathfrak{m}}+i\eta|}\frac{1}{\langle\xi-\omega\rangle^{2d}}\frac{1}{\langle\xi\rangle^{2d}}d\xi
\leq C \frac{|\log\eta|}{\langle\alpha\rangle},
\end{equation}
where the $\langle\alpha\rangle$ in the denominator again has to be
dropped if we remove the factor $\langle\xi\rangle^{2d}$.
\begin{proof}
We will first show that
\begin{equation}
\sup_{\omega} \int
\frac{1}{|\alpha+\xi^{\mathfrak{m}}+i\eta|}\frac{1}{\langle\xi-\omega\rangle^{2d}}d\xi
\leq C |\log\eta|.
\end{equation}
If $\eta\geq 1$, the integral is bounded by
\begin{equation}
\int
\frac{1}{|\alpha+\xi^{\mathfrak{m}}+i\eta|}\frac{1}{\langle\xi-\omega\rangle^{2d}}d\xi
\leq \frac{1}{\eta}\int\frac{1}{\langle\xi-\omega\rangle^{2d}}d\xi
\leq \frac{C}{\eta} \leq C|\log\eta|.
\end{equation}
Therefore, we only need to look at the case for $\eta<1$.
The integral over $|\alpha+\xi^{\mathfrak{m}}|\geq 1$ is bounded by
\begin{equation}
\int_{|\alpha+\xi^{\mathfrak{m}}|\geq
1}\frac{1}{|\alpha+\xi^{\mathfrak{m}}+i\eta|}\frac{1}{\langle\xi-\omega\rangle^{2d}}d\xi\leq
C\int \frac{1}{\langle\xi-\omega\rangle^{2d}}d|\xi|\leq C\leq
C|\log\eta|.
\end{equation}
The integral over $|\alpha+\xi^{\mathfrak{m}}|< 1$ can be estimated
by splitting the integration domain according to the size $k\leq
|\xi-\omega|\leq k+1$ for $k=0,1,\cdots$
\begin{equation}
\begin{aligned}
&\int_{|\alpha+\xi^{\mathfrak{m}}|<
1}\frac{1}{|\alpha+\xi^{\mathfrak{m}}+i\eta|}\frac{1}{\langle\xi-\omega\rangle^{2d}}d\xi\\
=&\sum_{k=0}^{+\infty}\int_{\{|\alpha+\xi^{\mathfrak{m}}|< 1\}\cap
\{k\leq |\xi-\omega|\leq
k+1\}}\frac{1}{|\alpha+\xi^{\mathfrak{m}}+i\eta|}
\frac{1}{\langle\xi-\omega\rangle^{2d}}d\xi\\
=&\sum_{k=0}^{+\infty}\int_{|\alpha+\xi^{\mathfrak{m}}|<
1}\frac{d|\xi|}{|\alpha+\xi^{\mathfrak{m}}+i\eta|}\int_{k\leq
|\xi-\omega|\leq k+1}
\frac{J(|\xi|,\theta_1,\cdots,\theta_{d-1})}{\langle\xi-\omega\rangle^{2d}}d\theta_1\cdots
d\theta_{d-1},
\end{aligned}
\end{equation}
where $J(|\xi|,\theta_1,\cdots,\theta_{d-1})$ denotes the Jacobian
to change to polar coordinates. For fixed $|\xi|$, we have the inequality
\begin{equation}
\begin{aligned}
\int_{k\leq |\xi-\omega|\leq k+1}
\frac{J(|\xi|,\theta_1,\cdots,\theta_{d-1})}{\langle\xi-\omega\rangle^{2d}}d\theta_1\cdots
d\theta_{d-1}\leq &\int_{k\leq |\xi-\omega|\leq k+1}
\frac{J(|\xi|,\theta_1,\cdots,\theta_{d-1})}{(1+k^2)^d}d\theta_1\cdots
d\theta_{d-1}\\
\leq&\left\{
\begin{array}{ll}
C\frac{(k+1)^{d-1}}{(1+k^2)^d}|\xi|^{\mathfrak{m-1}},& |\xi|\geq 1\\
C\frac{1}{(1+k^2)^d}|\xi|^{\mathfrak{m-1}},& |\xi|< 1.
\end{array}
\right.
\end{aligned}
\end{equation}
Summing up the terms over $k=0,1,\cdots$ gives
\begin{equation}
\int_{|\alpha+\xi^{\mathfrak{m}}|<
1}\frac{1}{|\alpha+\xi^{\mathfrak{m}}+i\eta|}\frac{1}{\langle\xi-\omega\rangle^{2d}}d\xi\leq
C\int_{|\alpha+\xi^{\mathfrak{m}}|<
1}\frac{|\xi|^{\mathfrak{m}-1}d|\xi|}{|\alpha+\xi^{\mathfrak{m}}+i\eta|}\leq
C|\log\eta|.
\end{equation}
We now prove \eqref{5.2.4}. Without loss of
generality, we may assume $|\alpha|\geq 2$. \\
The integral over the domain $|\alpha+\xi^{\mathfrak{m}}|\geq
{\alpha}/{2}$ is easily bounded by
\begin{equation}
\begin{aligned}
\int_{|\alpha+\xi^{\mathfrak{m}}|\le
\frac{|\alpha|}{2}}\frac{1}{|\alpha+\xi^{\mathfrak{m}}+i\eta|}\frac{1}{\langle\xi-\omega\rangle^{2d}}\frac{1}{\langle\xi\rangle^{2d}}d\xi
&\leq \frac{C}{|\alpha|}\int_{|\alpha+\xi^{\mathfrak{m}}|\le
\frac{|\alpha|}{2}}\frac{1}{\langle\xi-\omega\rangle^{2d}}\frac{1}{\langle\xi\rangle^{2d}}d\xi\\
&\leq \frac{C}{|\alpha|}\leq
\frac{C|\log\eta|}{\langle\alpha\rangle}.
\end{aligned}
\end{equation}
On the domain $|\alpha+\xi^{\mathfrak{m}}|\le |{\alpha}|/{2}$, we
have $\xi^{\mathfrak{m}}\geq |{\alpha}|/{2}$. Note also that
$\langle\xi\rangle^{2d}\geq \langle\xi^{\mathfrak{m}}\rangle$. The
integral over this domain is therefore bounded by
\begin{equation}
\begin{aligned}
\int_{|\alpha+\xi^{\mathfrak{m}}|\le
\frac{|\alpha|}{2}}\frac{1}{|\alpha+\xi^{\mathfrak{m}}+i\eta|}\frac{1}{\langle\xi-\omega\rangle^{2d}}\frac{1}{\langle\xi\rangle^{2d}}d\xi
&\leq
\frac{C}{\langle\alpha\rangle}\int_{|\alpha+\xi^{\mathfrak{m}}|\le
\frac{|\alpha|}{2}}\frac{1}{|\alpha+\xi^{\mathfrak{m}}+i\eta|}\frac{1}{\langle\xi-\omega\rangle^{2d}}d\xi\\
&\leq \frac{C|\log\eta|}{\langle\alpha\rangle}.
\end{aligned}
\end{equation}
\end{proof}
\end{lemma}

\begin{lemma}
Let
\begin{equation}
(Y_z
\hat{R})(\xi):=\int\frac{\hat{R}(\xi-y)}{z-\xi^{\mathfrak{m}}}d\xi
\end{equation}
be a family of linear operators parametrized by a complex parameter
$z=\alpha+i\eta$ with $\eta>0$.\\
(i) For $d\geq 3$ we have
\begin{equation}
\label{5.4.1} |Y_z \hat{R}|\leq C\|\hat{R}\|_{2d,2d}.
\end{equation}
(ii) If $z'=\alpha'+i\eta'$ and $\eta\geq\eta'$, then for $d\geq 3$
\begin{equation}
\label{5.4.2} |Y_z \hat{R}-Y_{z'} \hat{R}|\leq
C|z-z'|\eta^{-(1-\frac{1}{\mathfrak{m}})}\|\hat{R}\|_{2d,2d}.
\end{equation}
\end{lemma}

\begin{lemma}
Let $\hat{R}$ satisfy the smoothness condition defined in Section 1. We have the following inequality:
\begin{equation}
\label{homoin}
\varepsilon^{\mathfrak{m}}\int (\frac{1}{\xi_1^{\mathfrak{m}}-\varepsilon^{\mathfrak{m}}\xi^{\mathfrak{m}}}\wedge \frac{1}{\varepsilon^{\mathfrak{m}}}) \frac{\hat{R}(\xi_1-\varepsilon\xi)}{\xi_1^{\mathfrak{m}}}d\xi_1
\leq C\max\{\varepsilon^{\mathfrak{m}},\varepsilon^{d-\mathfrak{m}}|\log\varepsilon|(\xi^{\mathfrak{m}}+1)^{\frac{d}{\mathfrak{m}}-2},
\varepsilon^{d-\mathfrak{m}}(\xi^{\mathfrak{m}}+1)^{\frac{d}{\mathfrak{m}}-1}\}.
\end{equation}
\begin{proof}
We decompose the integral into three integrals on $\{\xi_1^{\mathfrak{m}}\geq\varepsilon^{\mathfrak{m}}\xi^{\mathfrak{m}}+1\}$, $\{\varepsilon^{\mathfrak{m}}(\xi^{\mathfrak{m}}+1)\leq\xi_1^{\mathfrak{m}}\leq \varepsilon^{\mathfrak{m}}\xi^{\mathfrak{m}}+1\}$ and $\{0\leq\xi_1^{\mathfrak{m}}\leq\varepsilon^{\mathfrak{m}}(\xi^{\mathfrak{m}}+1)\}$. Clearly, the first integral in bounded by $\varepsilon^{\mathfrak{m}} \int \frac{\hat{R}(\xi_1-\varepsilon\xi)}{\xi^{\mathfrak{m}}}d\xi_1$. For the second integral, note that $\hat{R}$ is bounded, and use $\frac{1}{\xi_1^{\mathfrak{m}}-\varepsilon^{\mathfrak{m}}\xi^{\mathfrak{m}}}$ to bound the term in the bracket on the left hand side of \eqref{homoin}. We then change variable to polar coordinates and let $Q=\xi_1^{\mathfrak{m}}$ to obtain
\begin{equation}
\begin{aligned}
\varepsilon^{\mathfrak{m}}\int_{\{\varepsilon^{\mathfrak{m}}(\xi^{\mathfrak{m}}+1)\leq\xi_1^{\mathfrak{m}}\leq \varepsilon^{\mathfrak{m}}\xi^{\mathfrak{m}}+1\}}  (\frac{1}{\xi_1^{\mathfrak{m}}-\varepsilon^{\mathfrak{m}}\xi^{\mathfrak{m}}}\wedge \frac{1}{\varepsilon^{\mathfrak{m}}}) \frac{\hat{R}(\xi_1-\varepsilon\xi)}{\xi_1^{\mathfrak{m}}}d\xi_1
&\leq C\varepsilon^{\mathfrak{m}}\int_{\varepsilon^{\mathfrak{m}}(\xi^{\mathfrak{m}}+1)}^{\varepsilon^{\mathfrak{m}}\xi^{\mathfrak{m}}+1}
\frac{Q^{\frac{d}{\mathfrak{m}}-2}}{Q-\varepsilon^{\mathfrak{m}}\xi^{\mathfrak{m}}}dQ\\
&\leq C\varepsilon^{d-\mathfrak{m}}|\log\varepsilon|(\xi^{\mathfrak{m}}+1)^{\frac{d}{\mathfrak{m}}-2}.
\end{aligned}
\end{equation}
For the third part, we use ${1}/{\varepsilon^{\mathfrak{m}}}$ to bound the term in the bracket on the left hand side of \eqref{homoin} and apply the same type of change of variable to obtain
\begin{equation}
\begin{aligned}
\varepsilon^{\mathfrak{m}}\int_{\{0\leq\xi_1^{\mathfrak{m}}\leq\varepsilon^{\mathfrak{m}}(\xi^{\mathfrak{m}}+1)\}} (\frac{1}{\xi_1^{\mathfrak{m}}-\varepsilon^{\mathfrak{m}}\xi^{\mathfrak{m}}}\wedge \frac{t}{\varepsilon^{\mathfrak{m}}}) \frac{\hat{R}(\xi_1-\varepsilon\xi)}{\xi_1^{\mathfrak{m}}}d\xi_1
&\leq t\int_0^{\varepsilon^{\mathfrak{m}}(\xi^{\mathfrak{m}}+1)} Q^{\frac{d}{\mathfrak{m}}-2}dQ\\
&\leq C\varepsilon^{d-\mathfrak{m}}(\xi^{\mathfrak{m}}+1)^{\frac{d}{\mathfrak{m}}-1}.
\end{aligned}
\end{equation}
\end{proof}
\end{lemma}

\section*{Acknowledgment} The authors would like to thank L{\'a}szl{\'o} Erd\H{o}s for useful discussions during the preparation of the manuscript. This work was funded in part by the National Science Foundation under Grant DMS-1108608 and the Air Force Office of Scientific Research under Grant NSSEFF-FA9550-10-1-0194.


\end{document}